\documentclass[11pt, reqno]{article}
\usepackage{textcmds} 
\usepackage{amsmath, amssymb, amsfonts, amstext, verbatim, amsthm, mathrsfs}
\usepackage{microtype}
\usepackage[all]{xy}
\usepackage[modulo]{lineno}
\usepackage[dvipsnames]{xcolor}
\usepackage{aliascnt}
\usepackage{enumitem}
\usepackage{xspace}
\usepackage{amsfonts}
\usepackage{amssymb}
\usepackage{amsthm}
\usepackage{rotating}
\usepackage[margin=3.5cm]{geometry}
\usepackage{dsfont}
\usepackage{bm}
\usepackage{subfigure}
\usepackage{amsmath}
\usepackage{array}
\usepackage[all]{xy}
\usepackage{euscript}
\usepackage[T1]{fontenc}
\usepackage{mathbbol}
\usepackage{nccmath}
\usepackage{marginnote}
\usepackage{stmaryrd}
\usepackage{youngtab}
\usepackage{tikz}
\usepackage{blkarray}
\usepackage{chngcntr}
\usepackage{multicol}
\usepackage{cancel}
\usepackage{mathabx}
\usepackage{titlesec}
\usepackage{caption}
\usepackage{tikz}
\usepackage[utf8]{inputenc}  
\usetikzlibrary{calc}
\usepackage{graphics,graphicx}  

\usepackage[colorlinks=true,linkcolor=black,citecolor=magenta,urlcolor=blue,citebordercolor={0 0 1},urlbordercolor={0 0 1},linkbordercolor={0 0 1}]{hyperref} 

\tikzset{node distance=3cm, auto}

\makeatother

\newtheorem{thm}{Theorem}[subsection]

\newtheorem{axiom}[thm]{Axiom}

\newtheorem{problem}[thm]{Problem}

\newtheorem{disclaimer}[thm]{Disclaimer}

\newtheorem{rmk}[thm]{Remark}
\newtheorem{example}[thm]{Example}

\theoremstyle{remark}
\numberwithin{equation}{section} 
\numberwithin{table}{section}
\numberwithin{figure}{section}
\numberwithin{thm}{section}

\newcommand{\NI}{{\noindent}}

\newcommand{\op}[1]{{\operatorname{#1}}}
\newcommand{\C}{\mathbb{C}}
\newcommand{\R}{\mathbb{R}}
\newcommand{\bdy}{\partial}
\newcommand{\ovl}{\overline}
\newcommand{\calM}{\mathcal{M}}
\newcommand{\al}{\alpha}
\newcommand{\be}{\beta}
\newcommand{\wh}{\widehat}
\newcommand{\Li}{\mathcal{L}_\infty}
\newcommand{\Si}{\Sigma}
\newcommand{\ra}{\rightarrow}
\newcommand{\la}{\lambda}
\newcommand{\Ga}{\Gamma}
\newcommand{\ga}{\gamma}
\definecolor{darkmagenta}{rgb}{0.55, 0.0, 0.55}
\newcommand{\hl}[1] {{\em #1}}
\newcommand{\uvl}{\underline}
\newcommand{\sss}{\vspace{2.5 mm}}
\newcommand{\ind}{\op{ind}}
\newcommand{\cz}{\op{CZ}}

\newcommand{\sht}{\mathfrak{s}}
\newcommand{\lng}{\mathfrak{l}}
\newcommand{\Z}{\mathbb{Z}}
\newcommand{\Q}{\mathbb{Q}}
\newcommand{\vir}{\op{vir}}
\newcommand{\nil}{\varnothing}
\newcommand{\cyl}{\op{cyl}}
\newcommand{\cha}{\op{CHA}}
\newcommand{\lan}{\langle}
\newcommand{\ran}{\rangle}
\newcommand{\ka}{\kappa}
\newcommand{\lin}{\op{lin}}
\newcommand{\aug}{\epsilon}

\newcommand{\comb}{\op{comb}}
\newcommand{\calP}{\mathcal{P}}
\newcommand{\calA}{\mathcal{A}}
\newcommand{\sym}{\op{Sym}}
\newcommand{\frakB}{\mathfrak{B}}
\renewcommand{\ll}{\llbracket}
\newcommand{\rr}{\rrbracket}
\newcommand{\hh}{\mathbb{h}}
\newcommand{\HH}{\mathbb{H}}
\newcommand{\rsft}{\op{RSFT}}
\newcommand{\ff}{\mathbb{f}}
\newcommand{\frakW}{\mathfrak{W}}
\newcommand{\sft}{\op{SFT}}
\newcommand{\qonly}{\op{q-only}}
\newcommand{\chlin}{{\op{CH}_{\op{lin}}}}
\newcommand{\chalin}{{\cha_\lin}}
\newcommand{\rsftlin}{{\rsft_\lin}}
\newcommand{\sftlin}{{\sft_\lin}}

\newcommand{\CP}{{\mathbb{CP}}}

\def\eps{{\varepsilon}}

\newcommand{\sh}{\op{SH}}
\renewcommand{\sc}{\op{SC}}
\newcommand{\wc}{\widecheck}
\newcommand{\wt}{\widetilde}
\newcommand{\morse}{\op{Morse}}

\newcommand{\FF}{\mathbb{F}}
\newcommand{\om}{\omega}

\newcommand{\mm}{\mathfrak{m}}
\newcommand{\ev}{\op{ev}}
\renewcommand{\lll}{\Langle}
\newcommand{\rrr}{\Rangle}
\newcommand{\T}{\mathcal{T}}
\renewcommand{\setminus}{\smallsetminus}
\newcommand{\calL}{\mathcal{L}}
\newcommand{\La}{\Lambda}

\newcommand{\TT}{\mathbb{T}}
\newcommand{\calS}{\mathcal{S}}

\makeatother

\title{Symplectic field theory: an overview}

\author{Richard Hind\thanks{RH thanks the Simons Foundation for their support under grant no. 633715.}\hspace{0.5em}  and Kyler Siegel\thanks{K.S. is partially supported by NSF grant DMS-2105578.}}

\date{\today}

\begin{document}

\maketitle

\begin{abstract}

We summarize some of the main ideas and results around symplectic field theory, from its early inception up to recent and ongoing developments.
\end{abstract}

\section*{Acknowledgements} This article is dedicated to Yasha Eliashberg with great admiration. 

\setcounter{tocdepth}{1}
\tableofcontents

\section{SFT at a glance}\label{sec:glance}

Symplectic field theory is a highly ambitious project which first appeared in crystallized form around 2000 in the work of Eliashberg--Givental--Hofer \cite{EGH2000} (see also Eliashberg's 2006 ICM address \cite{eliashberg_sft_and_applications}).
At its core, it is a machine which associates algebraic invariants to contact manifolds and symplectic cobordisms between them. These invariants are defined by packaging together counts of punctured pseudoholomorphic curves in symplectic manifolds with infinite ends, with each end typically modeled on the positive or negative half of the symplectization of a contact manifold, and with our curves asymptotic at each puncture to a Reeb orbit in the corresponding contact manifold.
Thus each puncture is positively or negatively asymptotically cylindrical in the target symplectic manifold, with the positive punctures serving as inputs and the negative punctures serving as outputs.

There are various different layers of the theory, corresponding roughly to whether we restrict to genus zero Riemann surfaces or allow all genera, and how many positive and negative punctures we permit.
The algebraic structures which arise from SFT are quite intricate and naturally reflect the compactification structure of the corresponding moduli spaces of punctured curves.
A basic familiar complication is that our curve counts are typically not invariant or meaningful on the nose, but rather constitute a kind of chain complex (or higher algebraic structure) which is independent of choices up to chain homotopy, so that the associated homology groups are robust invariants.
One of the simplest layers is \hl{linearized contact homology}, which heuristically counts only cylinders (or more precisely cylinders with extra capped punctures called ``anchors''), and which already encodes very rich symplectic and contact geometric data, but also presents plenty of technical and computational challenges.
Near the other extreme lies ``full SFT'', which incorporates curves of all genus and any numbers of positive and negative punctures, and whose scope is only beginning to be understood.

Some striking early applications of SFT include distinguishing contact manifolds and Legendrian submanifolds whose classical topological invariants coincide. 
In fact, many such results require only linearized contact homology or the so-called \hl{contact homology algebra} (and their Legendrian cousins -- see \S\ref{subsec:relSFT}), which involve only curves of genus zero and one positive end, and which serve as an important precursor to full SFT (see e.g. \cite{eliashberg_contact,Chekanov_DGA,ustilovsky1999infinitely,yau2004cylindrical,van2004contact,Bourgeois_MB}).
However, the range of applicability of SFT extends much further, to things like existence of Reeb orbits, ruling out symplectic fillings and symplectic cobordisms, quantitative symplectic and contact nonsqueezing, and beyond.
Although we cannot possibly do justice to all known and expected applications of SFT, we will describe one simple appealing consequence from \cite[\S1.7]{EGH2000} in \S\ref{sec:applications}. 


The name ``symplectic field theory'' reflects the fact that, in the spirit of topological quantum field theory \cite{atiyah1988topological}, we have a functor from a geometric category (consisting of contact manifolds and symplectic cobordisms between them) to an algebraic category (in the simplest case vector spaces and linear maps between them). 
Crucially, given two symplectic cobordisms such that the positive end of the first and the negative end of the second are modeled on the same contact manifold, we can then concatenate them together to get a new symplectic cobordism  whose associated algebraic invariants are given by composing (in a suitable sense) those of the two given cobordisms.
Said differently, we can decompose a symplectic manifold along a contact hypersurface into two symplectic cobordisms by a process called \hl{neck stretching}, and this reduces the computation of algebraic invariants for the initial space into those for two potentially simpler pieces.
Although pseudoholomorphic curve invariants tend to be quite global in nature, this gives a powerful source of semi-local reduction, which applies even for closed curves in closed symplectic manifolds (indeed, Gromov--Witten theory can be thought of as a special case of SFT for symplectic cobordisms with no positive or negative ends).

For example, we can decompose the complex projective plane $\mathbb{CP}^2$ (with its Fubini--Study symplectic form) along the contact hypersurface $S^{3}$ given by the boundary of a small tubular neighborhood of the line at infinity. This results in two pieces: (a) $\mathbb{C}^2$ and (b) the total space of the line bundle $\mathcal{O}(1) \rightarrow \mathbb{CP}^1$, where the former has a positive end modeled on the standard contact $S^{3}$ and the latter has a negative end modeled on the same contact manifold. This decomposes the Gromov--Witten invariants of $\mathbb{CP}^2$ into SFT invariants of $\mathbb{C}^2$ and $\mathcal{O}(1)$. The bundle structure on the latter makes it fairly easy to enumerate its punctured curves, and with a little bit of effort we recover the celebrated Caporaso--Harris recursive formula \cite{CaH} for Severi degrees of the projective plane.\footnote{These are roughly the number of complex algebraic plane curves of a given genus which pass through an appropriate number of generic points. See also \cite[\S15.1]{IP2004symplecticsum} for a closely related picture using the language of relative Gromov--Witten theory.}

The rest of this note is structured as follows. We begin in \S\ref{sec:historical} with some recollections (based on conversations with Yasha Eliashberg) around the historical development of symplectic field theory. In \S\ref{sec:compactness}, we recall the SFT compactness theorem, which is a key ingredient to getting the theory off the ground. 
In \S\ref{sec:transversality} we briefly address the technical issue of transversality.
We then introduce the algebraic formalism of SFT in \S\ref{sec:formalism}, and discuss applications in \S\ref{sec:applications}.
Finally, in \S\ref{sec:further} we mention various extensions of the theory, some of which have already appeared in the literature, and others of which are more speculative. 

Let us emphasize that this note is only a biased impressionistic sketch of symplectic field theory, and barely scratches the surface of the literature. In particular, we neglect to mention many important results on foundations, computations, and applications (some of which appear elsewhere in this volume), and our attributions are no by means exhaustive.
For more a comprehensive introduction to the theory, we refer the reader to the original papers \cite{EGH2000,eliashberg_sft_and_applications} and the references therein, as well as Wendl's excellent notes \cite{wendl_SFT_notes}.


\section{Historical recollections}\label{sec:historical}

We begin by setting the scene for the discovery and early development of symplectic field theory.
The section is essentially a summary of a conversation which took place between Yasha Eliashberg and the two authors at the Institut Mittag-Leffler during summer 2024 (any inaccuracies are surely due to the present authors).

In the early 1980s, Eliashberg was already talking informally with Viatcheslav Kharlamov about the possibility of applying holomorphic methods to four-manifold topology. Errett Bishop in 1965 \cite{bishop_65} had shown that a neighborhood of an elliptic complex tangency point in a (real) two-dimensional surface $S \subset \C^2$ can be foliated by boundaries of holomorphic disks. If such local families of disks could somehow be extended to form three-dimensional Levi flat hypersurfaces, there would clearly be strong implications for the isotopy classes of surfaces. The breakthrough came in a 1983 paper of Eric Bedford and Bernard Gaveau \cite{bedford1983envelopes}, with their main theorem showing that in certain circumstances a two-sphere $S \subset \C^2$ does indeed bound a Levi flat ball. Let $(z,w)$ be coordinates on $\C^2$. The paper \cite{bedford1983envelopes} assumes that $S$ is a graph over a two-sphere $\overline{S} \subset \{ \mathrm{Im}(w)=0 \}$, that $\{ (z,w) \;|\; (z, \mathrm{Re}(w)) \in \overline{S} \}$ is strictly pseudoconvex, and that $S$ has exactly two complex tangency points. Then the Bishop families extend to form a Levi flat ball $B$ with $\partial B = S$. Eliashberg realized the graphical hypothesis could be removed, and the technique extended to show that two-spheres in smooth boundaries of strictly pseudoconvex domains in Stein manifolds bound balls which are foliated by holomorphic disks.

Around the same time, Daniel Bennequin \cite{bennequin_pfaff} proved Thurston's conjecture on transverse knots in the standard contact $\R^3$, and as a consequence established the existence of nonstandard contact structures on the three-sphere. Bennequin's proof involves intricate knot theory; as an early indication of the power of holomorphic methods, Eliashberg showed that the result follows readily from the existence of fillings by holomorphic disks.

Eliashberg wrote to Gromov about these results. This was before the appearance of his pseudoholomorphic curve theory, but Gromov replied that he was also thinking about these topics, and that likely the general context should be contact manifolds bounding symplectic manifolds.

Eliashberg worked as a computer programmer in Leningrad from 1980 until emigrating to the US in 1988. In this period he had little time for mathematics, but was excited to return to work on symplectic topology, and in particular holomorphic disks, initially at MSRI and then after settling at Stanford. Pseudoholomorphic curves had now been introduced to symplectic topology, so theorems could apply in contact and symplectic settings.

One result was a proof of Cerf's theorem that diffeomorphisms of the three-sphere extend to the four-ball. The standard contact structure on $S^3$ arises naturally as the complex tangencies in the boundary of the four-ball $B^4 \subset \C^2$. Hence  two-spheres in $S^3$ can be filled by holomorphic disks mapping to $B^4$. In fact, using coordinates $(z,w)$ as above, each of the two-spheres $S_c := \{ \mathrm{Im}(w) = c \} \subset S^3$, for $c \in (-1,1)$, has two elliptic points and bounds the three-ball $\{ \mathrm{Im}(w) = c \} \subset B^4$, which is foliated by the holomorphic disks $\{ \mathrm{Re}(w) = d, \, \mathrm{Im}(w) = c \} \subset B^4$ for $d \in (-\sqrt{c}, \sqrt{c})$. Now, by Eliashberg's classification of contact structures on $S^3$, a diffeomorphism $\phi$ of $S^3$ is isotopic to a contactomorphism $\psi$ of the standard contact structure. A contactomorphism maps the spheres $S_c$ to two-spheres which also have two elliptic points, and hence the $\psi(S_c)$ can also be filled by holomorphic disks. The proof proceeds to extend $\psi$ over $B^4$ by extending $\psi |_{S_c}$ over these filling disks.

Moving to more general cases, a key requirement for arguments of this kind is that our contact manifold be fillable, that is, it appears as the boundary of a compact symplectic manifold with a suitable compatibility between the contact and symplectic structures. In general, all we can say is that a contact manifold sits as a contact type hypersurface in its (non-compact) symplectization. In the early 1990s, Eliashberg worked with Helmut Hofer, attempting to
apply holomorphic disk techniques in contact geometry.
The breakthrough was Hofer's proof of the Weinstein conjecture for $S^3$, and also for overtwisted contact three-manifolds \cite{Hofer_pseudoholomorphic_curves_in_symplectizations}. Hofer's insight was that a family of holomorphic disks (say with boundary on a fixed sphere in a contact hypersurface in its symplectization) either has a convergent subsequence, or, looking at points where the gradient explodes, we can extract a sequence of holomorphic maps converging to a holomorphic plane which is asymptotic to a closed Reeb orbit. This is perhaps somehow reminiscent of Gromov's compactness theorem, where a holomorphic sphere may bubble from a sequence of closed curves. In any case, the natural relation between holomorphic curves and closed Reeb orbits was now established.

Very quickly, Eliashberg and Hofer realized there must be a rich algebraic structure for holomorphic curves in symplectic cobordisms with contact type boundaries. Now, instead of the closed curves of Gromov-Witten theory, we should study maps from Riemann surfaces with punctures, asymptotic as we approach the punctures to closed Reeb orbits on the boundary. The symplectization case includes contact homology, which appears in Eliashberg's ICM article \cite{eliashberg_contact}.

Eliashberg went on to consider the relative case, and invariants of Legendrian knots. Similar invariants for Legendrian knots in $\R^3$ were constructed at the same time by Yuri Chekanov \cite{Chekanov_DGA}. Chekanov's invariants were rigorously defined using combinatorial methods, but were inspired by the emerging holomorphic curve picture; indeed, Chekanov's differential counts immersed polygons in the Lagrangian projection of the knot, which correspond to holomorphic curves in the symplectization. The domains of our holomorphic curves are now disks with boundary punctures. The boundary projects to the Legendrian in the contact manifold, and the punctures are asymptotic to Reeb chords. These invariants can be used to distinguish Legendrian knots with the same ``classical'' invariants, namely the topological knot type, the Thurston-Bennequin invariant and the rotation number.

Eliashberg describes his meetings with Alexander Givental as very important for the development of the subject. Conversations with Hofer had already considered possible higher algebraic invariants extending contact homology. Givental recognised the Poisson algebra structure present when considering curves of genus zero, and in multiple conversations they worked out the correct formalism for much of the theory. The famous SFT paper \cite{EGH2000} soon followed, with characteristic contributions from each of the three authors.

At the time, it appeared that a compactness theorem would be the main input from geometric analysis required for a rigorous theory (transversality issues were not viewed as very serious, at least by Eliashberg). Eliashberg was working on such a compactness result with Fr\'ed\'eric Bourgeois when he learned that Hofer, Krzysztof Wysocki and Eduard Zehnder were collaborating on the same project. The seminal foundational paper on SFT compactness subsequently appeared in the joint five author paper \cite{BEHWZ}.

\section{SFT compactness theorem}\label{sec:compactness}

Let $\calM_{g,k}$ denote the moduli space of biholomorphism classes of genus $g$ Riemann surfaces with $k$ ordered marked points, and let $\ovl{\calM}_{g,k}$ denote its Deligne--Mumford compactification.
Recall that an element of $\ovl{\calM}_{g,k}$ is a nodal Riemann surface of genus $g$ with $k$ marked points which is stable in the sense that each component has negative Euler characteristic after removing all of its marked points and nodal points.

Given an almost complex manifold $(X^{2n},J)$ and homology class $A \in H_2(X)$, we can consider the moduli space $\calM_{g,k,A}^{X,J}$ of all $J$-holomorphic maps $u: \Si \ra X$ in homology class $A$, with domain Riemann surface varying over $\Si \in \calM_{g,k}$, modulo biholomorphic reparametrizations.
One of Gromov's key insights in \cite{gromov1985pseudo} is that when $X$ is compact and $J$ tames a symplectic form on $X$, the moduli space $\calM_{g,k,A}^{X,J}$ also has a natural compactification $\ovl{\calM}_{g,k,A}^{X,J}$ by what are now called \hl{stable maps}. Thus an element of $\ovl{\calM}_{g,k,A}^{X,J}$ is a $J$-holomorphic map from a nodal Riemann surface of genus $g$ with $k$ marked points into $X$ which lies in homology class $A$ and is stable in the sense that each {\em constant} component has negative Euler characteristic after removing all of its marked points and nodal points.

The SFT compactness theorem \cite{BEHWZ} extends Gromov's compactification by allowing the target space $X$ to be noncompact and the domain Riemann surface $\Si$ to have punctures.
There are several variants of the SFT compactness theorem, but in a typical setting the target space is a completed symplectic cobordism of the form 
\begin{align*}
\wh{X} = (\R_{\leq 0} \times Y_-) \cup X \cup (\R_{\geq 0} \times Y_+),
\end{align*}
where 
\begin{itemize}
	\item $X^{2n}$ is a Liouville cobordism with positive contact boundary $Y_+$ and negative contact boundary $Y_-$ (that is, $X$ carries a one-form $\la$ such that $d\la$ is symplectic and $\la$ restricts to a positive contact form $\al_+$ on $Y_+$ and a negative contact form $\al_-$ on $Y_-$)
	\item $\wh{X}$ carries the symplectic form given by $d\la$ on $X$, $d(e^r\al_+)$ on $\R_{\geq 0} \times Y_+$, and $d(e^r\al_-)$ on $\R_{\leq 0} \times Y_-$ (here $r$ is the coordinate on $\R_{\leq 0},\R_{\geq 0}$)
	\item $\wh{X}$ also carries a tame almost complex structure $J$ which is SFT admissible, meaning roughly that on the ends it is translation invariant, preserves the contact planes, and maps the cylindrical direction $\bdy_r$ to the Reeb direction. 
\end{itemize}
We also often assume that the Reeb orbits of $(Y_\pm,\al_\pm)$ are nondegenerate, which can always be achieved by a small perturbation. Recall that by definition the Reeb orbits of a contact manifold $Y$ with contact form $\al$ are the periodic trajectories of the Reeb vector field $R_\al$, which is characterized by $d\al(R_\al,-) = 0$ and $\al(R_\al) = 1$, and nondegeneracy implies in particular that there are only finitely many Reeb orbits with action (i.e. period) satisfying a given upper bound.
In \S\ref{sec:further} we will discuss various relaxations of the above assumptions.

Given tuples of Reeb orbits $\Ga_+ = (\ga_1^+,\dots,\ga_{s_+}^+)$ in $Y_+$ and $\Ga_- = (\ga_1^-,\dots,\ga_{s_-}^-)$ in $Y_-$, let $\calM^{\wh{X},J}_{g,k}(\Ga_+,\Ga_-)$ denote the moduli space of $J$-holomorphic maps $u: \Si \ra \wh{X}$, where:  
\begin{itemize}
	\item $\Si$ is a Riemann surface of genus $g$ with $k$ ordered marked points and $s_+ + s_-$ ordered punctures (we call the first $s_+$ punctures positive and the last $s_-$ negative)
	\item for $i = 1,\dots,s_+$, $u$ is positively asymptotic at the $i$th positive puncture to the Reeb orbit $\ga_i^+$ in $Y_+$, which means roughly that the $Y_+$ component of $u$ limits to a parametrization of $\ga_i$ as we approach the $i$th puncture, while the $\R_{\geq 0}$ component of $u$ tends to $+\infty$ 
	\item similarly, for $j = 1,\dots,s_-$, $u$ is negatively asymptotic at the $j$th negative puncture to the Reeb orbit $\ga_i^-$ in $Y_-$. 
\end{itemize}
Note that in particular the map $u: \Si \ra \wh{X}$ is proper. We will refer to such a curve with positive and negative punctures asymptotic to Reeb orbits as \hl{asymptotically cylindrical}.

The SFT compactness theorem states that $\calM^{\wh{X},J}_{g,k}(\Ga_+;\Ga_-)$ has a natural compactification $\ovl{\calM}^{\wh{X},J}_{g,k}(\Ga_+;\Ga_-)$ by so-called \hl{stable pseudoholomorphic buildings}. 
It first appeared in \cite{BEHWZ}, building on Hofer's pioneering work \cite{Hofer_pseudoholomorphic_curves_in_symplectizations} on punctured curves and the Weinstein conjecture (see also the alternative approach in \cite{CM_SFT_compactness} and the textbook \cite{abbas2014introduction}).
Roughly speaking, a stable pseudoholomorphic building in $\ovl{\calM}^{\wh{X},J}_{g,k}(\Ga_+;\Ga_-)$ consists of
\begin{itemize}
	\item some number (possibly zero) of levels in the symplectization $\R \times Y_-$
	\item a level in $\wh{X}$
	\item some number (possibly zero) of levels in the symplectization $\R \times Y_+$,
\end{itemize}
where
\begin{itemize}
	\item each level is comprised of a nodal asymptotically cylindrical marked curve with possibly disconnected domain
	\item the levels are ordered vertically, so that for any two adjacent levels the negative asymptotic Reeb orbits of the upper level agree with the positive asymptotic Reeb orbits of the lower level
	\item the symplectization levels are taken modulo the $\R$-action by translations in the target space
	\item the total domain after gluing paired punctures is a connected nodal surface of genus $g$ with $k$ marked points and $s_+ + s_-$ punctures
	\item the positive punctures at the topmost level are asymptotic to $\Ga_+$, and the negative punctures at the bottommost level are asymptotic to $\Ga_-$	
	\item the configuration is \hl{stable} in the sense that each nonconstant component has negative Euler characteristic after removing all marked points and nodal points, and also no symplectization level consists entirely of trivial cylinders over Reeb orbits.
\end{itemize}
See Figure~\ref{fig:SFT_compactness} for a cartoon.

\begin{figure}
  \includegraphics[scale=.7]{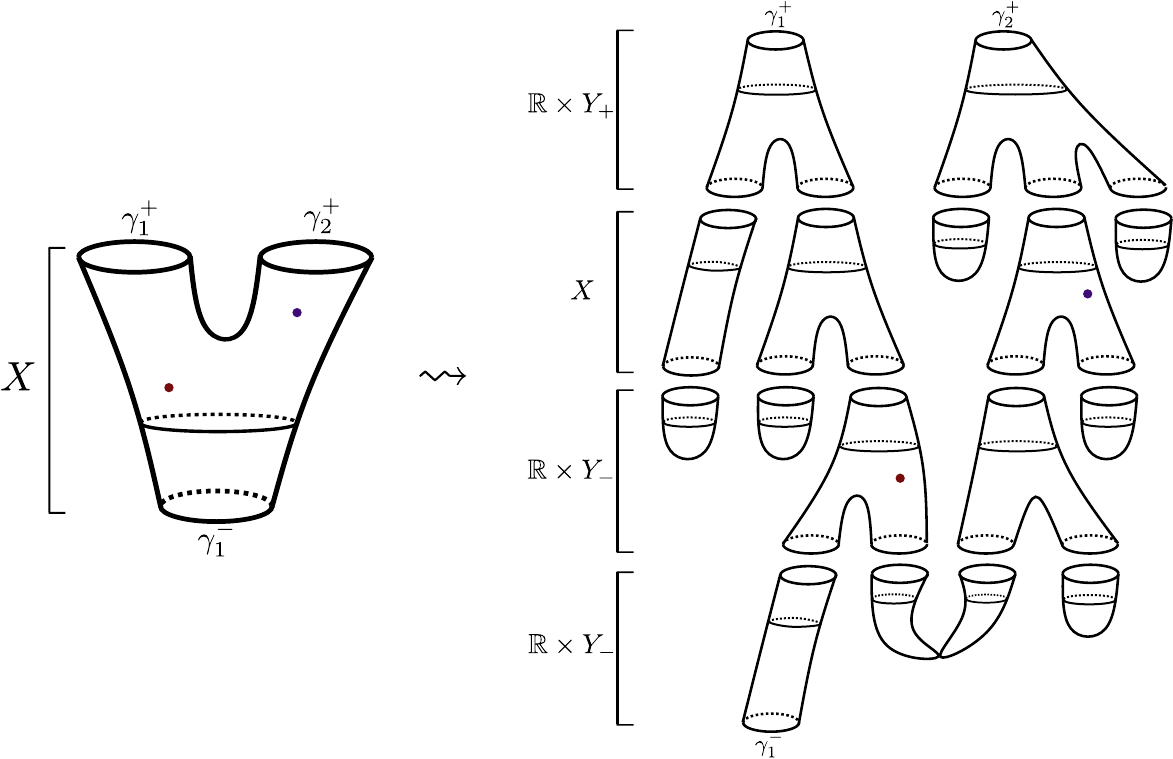}
  \caption{An asymptotically cylindrical pseudoholomorphic pair of pants and a stable pseudoholomorphic building to which it could a priori degenerate under the SFT compactness theorem.}
  \label{fig:SFT_compactness} 
\end{figure}

It is sometimes useful to slightly refine the above by taking homology classes of curves into account (this becomes essential in the non-exact case as in \S\ref{subsec:nonexact}).
Let $H_2(X,\Ga_+ \cup \Ga_-)$ denote the homology group of integral $2$-chains $Z$ in $X$ satisfying $\partial Z = \sum\limits_{i=1}^{s_+} \ga_i^+ - \sum\limits_{j=1}^{s_-} \ga_j^-$, modulo boundaries of $3$-chains (this forms a torsor over the usual integral homology group $H_2(X)$). 
By identifying $\wh{X}$ diffeomorphically with the interior of $X$, each curve in $\calM^{\wh{X},J}_{g,k}(\Ga_+;\Ga_-)$ has an associated homology class $[u] \in H_2(X,\Ga_+ \cup \Ga_-)$.
For fixed $A \in H_2(X,\Ga_+ \cup \Ga_-)$, we consider the subspace $\calM^{\wh{X},J}_{g,k,A}(\Ga_+;\Ga_-) \subset \calM^{\wh{X},J}_{g,k}(\Ga_+;\Ga_-)$ of those curves $u: \Si \ra \wh{X}$ with $[u] = A$, along with its compactification $\ovl{\calM}^{\wh{X},J}_{g,k,A}(\Ga_+;\Ga_-) \subset \ovl{\calM}^{\wh{X},J}_{g,k}(\Ga_+;\Ga_-)$ consisting of those stable pseudoholomorphic buildings such that the total glued curve lies in the homology class $A$.

\begin{rmk}
A key observation underlying Gromov's compactness theorem is that the \hl{energy} of a closed curve (essentially the $L^2$ norm of its derivative) agrees with its symplectic area, and hence is a priori bounded for curves lying in a fixed homology class $A$ (here it is crucial that the almost compact structure $J$ is tamed by the symplectic form on $X$).
For punctured curves in $\calM_{g,k}^{\wh{X},J}(\Ga_+;\Ga_-)$ as above, the energy in fact depends only on the asymptotic Reeb orbits $\Ga_+,\Ga_-$, which is why the moduli space $\ovl{\calM}^{\wh{X},J}_{g,k}(\Ga_+;\Ga_-)$ is compact without specifying any homology class. 
However, this relies on Stokes' theorem, and hence does not hold if we relax the assumption that the symplectic form $\wh{X}$ is exact (see \S\ref{subsec:nonexact}).
Incidentally, the naive notion of energy for asymptotically cylindrical curves is always infinite, but there is a natural replacement called the \hl{Hofer energy} (and similarly for the symplectic area) -- see \cite[\S5.3]{BEHWZ}.

In the exact case (i.e. for the completion of a Liouville domain or a symplectization of a contact manifold), a simple but important observation is that, given an asymptotically cylindrical curve with asymptotics $\Ga_+,\Ga_-$, the total action of $\Ga_+$ (i.e. the sum of the periods of its constituent Reeb orbits) minus the total action of $\Ga_-$ is always nonnegative. 
This follows from Stokes' theorem and the definition of SFT admissible almost complex structures.
In particular, this makes it possible to define an action filtration on SFT which is sensitive to quantitative information (see \S\ref{subsec:quant}).
\end{rmk}

In a typical usage of the SFT compactness theorem, one seeks to show that some moduli space $\calM^{\wh{X},J}_{g,k,A}(\Ga_+;\Ga_-)$ of expected dimension zero is a finite set by showing that it is compact, which follows if we can establish $\ovl{\calM}^{\wh{X},J}_{g,k,A}(\Ga_+;\Ga_-) = \calM^{\wh{X},J}_{g,k,A}(\Ga_+;\Ga_-)$, i.e. there are no nontrivial stable pseudoholomorphic buildings to which a curve in $\calM^{\wh{X},J}_{g,k,A}(\Ga_+;\Ga_-)$ could degenerate.
A priori there are many potentially elaborate buildings in $\ovl{\calM}^{\wh{X},J}_{g,k,A}(\Ga_+;\Ga_-)$ (recall Figure~\ref{fig:SFT_compactness}), but one observes that most of these have expected codimension at least one, and hence could be ruled out if we knew that every stratum appears with its expected codimension (this is the problem of transversality, which we take up in the next section).
Similarly, in the case that the moduli space $\calM^{\wh{X},J}_{g,k,A}(\Ga_+;\Ga_-)$ has expected dimension $1$, one typically seeks to show that its SFT compactification $\ovl{\calM}^{\wh{X},J}_{g,k,A}(\Ga_+;\Ga_-)$ is a one-dimensional cobordism whose boundary components corresponds to precisely two-level stable pseudoholomorphic buildings, and this would hold if we could rule out more complicated buildings.

\sss

There are several important variations on the above SFT compactness theorem which are crucial for constructing the full SFT package.
The first is where we replace $\wh{X}$ with the symplectization of a contact manifold $Y$, i.e. $\R \times Y$ equipped with the symplectic form $d(e^r\al)$, where $\al$ is a contact form on $Y$.
In this case we work with an almost complex structure which is SFT admissible for the symplectization $\R \times Y$, which in particular means globally translation invariant.
The corresponding SFT compactification then consists of stable pseudoholomorphic buildings with one or more symplectization levels $\R \times Y$, each of which is taken modulo $\R$-translations in the target space.
Note that both the uncompactified and compactified moduli spaces of curves in a symplectization inherit $\R$-actions induced by translations in the target space. 

Another variation is where we take a one-parameter family of almost complex structures $\{J_t\}_{t \in [0,1]}$ (or possibly a higher dimensional family), and we consider the parametrized moduli space of pairs $(u,t)$ such that $u$ is $J_t$-holomorphic.
Lastly, there is the degenerate case of the above which is relevant for neck-stretching, where $\{J_t\}_{t \in [0,1)}$ is a family of almost complex structures on $\wh{X}$ which approaches the neck-stretching limit as $t \ra 1$. 
This means that $X$ splits along a contact hypersurface $Y$ into two Liouville cobordisms $X_-,X_+$, and $J_t$ is cylindrical on an increasingly long collar neighborhood of $Y$.
In this case, the SFT compactification includes limiting buildings associated with $t = 1$ which consist of some number of symplectization levels $\R \times Y_-$, a cobordism level $\wh{X}_-$, some number of symplectization levels $\R \times Y$, a cobordism level $\wh{X}_+$, and some number of symplectization levels $\R \times Y_+$.

\sss

Finally, note that while the SFT compactness theorem provides a natural geometric prescription for compactifying moduli spaces of asymptotically cylindrical curves, for stronger control on the boundary structure of these compactified moduli spaces we also require counterpart gluing theorems (similar considerations hold e.g. for Morse and Floer homology). For example, in the case of a compactified one-dimensional moduli space we will need a gluing theorem stating that every two-level stable pseudoholomorphic building which a priori appears in $\ovl{\calM}^{\wh{X},J}_{g,k,A}(\Ga_+;\Ga_-)$ really is a limit of curves in the uncompactified space $\calM^{\wh{X},J}_{g,k,A}(\Ga_+;\Ga_-)$.
The proof structure of a gluing theorem for pseudoholomorphic curves is detailed in \cite[\S10]{JHOL} in the context of Gromov--Witten theory, while gluing theorems for asymptotically cylindrical curves with paired punctures are proved in \cite[\S5]{Pardcnct} in the context of the contact homology algebra (see also \cite{HuT,HTII}).
To our knowledge, the most general gluing theorem needed for symplectic field theory has not appeared in full detail in the literature, but is widely expected to proceed along similar lines to loc. cit.

\begin{rmk}
In the above discussion, we neglected to mention an extra piece of data, namely \hl{asymptotic markers}, which single out a preferred direction at each puncture of the domain Riemann surface. Although these do not affect the basic structure of the SFT compactification, they do become important when discussing gluing along multiply covered Reeb orbits, giving rise to extra combinatorial factors.
\end{rmk}

\begin{rmk}
Fish \cite{fish_target_local} has proven a ``target local'' version of Gromov's compactness which can often be applied to punctured curves in more general settings than the ones discussed above.
\end{rmk}


\section{Transversality}\label{sec:transversality}


Before discussing the algebraic formalism of SFT, we should mention the issue of transversality.
In order to read off nice algebraic relations from compactified moduli spaces of punctured curves, 
we would ideally like to know (among other things) that all relevant moduli spaces are smooth manifolds whose actual dimension agrees with the expected dimension, at least for a generic choice of almost complex structure. 
With the notation of \S\ref{sec:compactness}, the expected dimension of the uncompactified moduli space $\calM_{g,k,A}^{\wh{X},J}(\Ga_+;\Ga_-)$ is given by the Fredholm index
\begin{align}\label{eq:ind}
\ind\,\calM_{g,k,A}^{\wh{X},J}(\Ga_+;\Ga_-) = (n-3)(2-2g-s_- - s_+) + \sum_{i=1}^{s_+} \cz(\ga_i^+) - \sum_{j=1}^{s_-} \cz(\ga_j^-) + 2c_1(A) + 2k,
\end{align}
where $\dim \wh{X} = 2n$.
Here $\cz(\ga)$ denotes the Conley--Zehnder index, which measures the winding number of the contact hyperplanes around a (nondegenerate) Reeb orbit $\ga$, and $c_1(A)$ is a relative Chern number (both of these terms depend on auxiliary trivialization data for the contact hyperplanes, but the expression in \eqref{eq:ind} does not).
Typically one presents $\calM_{g,k,A}^{\wh{X},J}(\Ga_+;\Ga_-)$ as the set of zeroes of a certain Fredholm section of a Banach vector bundle over a Banach manifold (the section is essentially the Cauchy--Riemann operator), and if we can show that this section is transverse to the zero section, then it will follow by a Banach space version of the inverse function theorem that $\calM_{g,k,A}^{\wh{X},J}(\Ga_+;\Ga_-)$ is a smooth manifold of dimension equal to its Fredholm index. In this case we will say that the corresponding moduli space is {\em regular} (or ``transversely cut out'').

It turns out that transversality can indeed be arranged by a generic choice of $J$ for all {\em simple} curves, i.e. those which do not factor as $\Si \xrightarrow{f} \Si' \ra \wh{X}$ with $\Si'$ another Riemann surface and $f$ a holomorphic map of degree at least two.
Indeed, there is a by now standard method for achieving transversality for simple curves by generic perturbations of a given almost complex structure (see e.g. \cite[\S3]{JHOL}), and this applies also to moduli spaces of asymptotically cylindrical curves after some adaptations (see e.g. \cite[\S8]{wendl_SFT_notes}).
However, this result generally fails for multiply covered curves, which tend to appear unavoidably in families of greater than expected dimension, even for generic almost complex structures.

\begin{example}\label{ex:mult_cov}
Here is a simple concrete example which illustrates the failure of transversality for multiple covers.
Consider $X := E(1,c)\; \setminus\; \op{Int}\; E(1,1+\delta)$ for $c > 0$ very large and $\delta > 0$ very small, 
where $E(a,b) := \{(z_1,z_2)\;|\; \pi|z_1|^2/a + \pi|z_2|^2/b \leq 1\}$ denotes the four-dimensional symplectic ellipsoid in $\C^2$ with area factors $a,b \in \R_{>0}$.
Note that $X$ is a Liouville cobordism with positive boundary $Y_+ := \bdy E(1,c)$ and negative boundary $Y_- := \bdy E(1,1+\delta)$.
The Reeb orbits of $Y_\pm$ are $\sht_\pm := Y_\pm \cap (\C \times \{0\})$ and $\lng_\pm := Y_\pm \cap (\{0\} \times \C)$ and their multiple covers, and we can globally trivialize the contact hyperplane distribution such that, for all $k \in \Z_{\geq 1}$, the Reeb orbit in $Y_\pm$ of $k$th smallest action has Conley--Zehnder index $1+2k$.
With this trivialization, the relative first Chern number term in the
 index formula \eqref{eq:ind} vanishes, so the moduli space $\calM^{\wh{X},J}_{0,0}(\sht_+;\sht_-)$ of $J$-holomorphic cylinders  which are positively asymptotic to $\sht_+$ and negatively asymptotic to $\sht_-$ has expected dimension zero.
Moreover, it is possible to show (somewhat less trivially) that $\calM^{\wh{X},J}_{0,0}(\sht_+;\sht_-)$ is nonempty for any generic choice of SFT admissible almost complex structure $J$.

Similarly, letting $\sht_\pm^2$ denote the two-fold cover of the Reeb orbit $\sht_\pm$, the corresponding moduli space of cylinders $\calM^{\wh{X},J}_{0,0}(\sht_+^2;\sht_-^2)$ has expected dimension $-2$. 
Observe that this moduli space is necessarily nonempty for any generic SFT admissible $J$ (by taking two-fold covers of curves in $\calM^{\wh{X},J}_{0,0}(\sht_+;\sht_-)$), so evidently it cannot be a smooth manifold whose dimension matches its expected dimension.
Note that even if we are not directly interested in the moduli space $\calM^{\wh{X},J}_{0,0}(\sht_+^2;\sht_-^2)$, it may well spoil transversality for other moduli spaces we do care about by appearing in buildings in their SFT compactifications.
\end{example}


\sss

In order to overcome this difficulty, one idea is to introduce a wider class of ``abstract'' perturbations of the pseudoholomorphic curve equation which provide enough freedom to achieve transversality. 
For example, we could introduce an inhomogeneous term to the Cauchy--Riemann equation, which indeed suffices to achieve transversality locally near any given curve. 
However, it then becomes a quite subtle problem to make these perturbations in a coherent way in order to obtain globally defined moduli spaces which suitably respect the SFT compactification structure and the action by biholomorphic parametrizations.

Suppose that $X$ is a Liouville cobordism between contact manifolds $Y_+$ and $Y_-$, and let us pretend for a moment that we can find SFT admissible almost complex structures on $\wh{X}$ and $\R \times Y_\pm$ such that all relevant uncompactified  moduli spaces in $\wh{X}$ and $\R \times Y_\pm$ are regular, and moreover their compactifications have sufficiently nice boundary stratifications.
The basic structure coefficients of SFT should then come from the signed\footnote{Here the signs come from assigning coherent orientations to our moduli spaces as in \cite{bourgeois2004coherent} or \cite[\S1.8]{EGH2000} (see also \cite[\S11]{wendl_SFT_notes} and \cite{bao2023coherent}).} counts of points in moduli spaces of the form $\ovl{\calM}_{g,0,A}^{\wh{X},J}(\Ga_+;\Ga_-)$ and $\ovl{\calM}_{g,0,A}^{\R \times Y_\pm,J_\pm}(\Ga_+;\Ga_-) / \R$ for all choices of Reeb orbits $\Ga_\pm$ and homology classes $A$ such that these have expected dimension zero.
In particular, under our transversality assumption these should be finite $0$-dimensional manifolds which coincide with their uncompactified counterparts.
Moreover, the basic algebraic relations which these counts satisfy come from considering moduli spaces of the same form but of expected dimension one, for which the signed count of boundary points should vanish.

As the above transversality assumption is largely unrealistic (c.f. Example~\ref{ex:mult_cov}), here is a (somewhat vague) formulation of the problem we must solve in order to define SFT:
\begin{problem}\label{prob:SFT_trans}
Come up with a coherent framework for assigning counts $\#^\vir\ovl{\calM}_{g,0,A}^{\wh{X},J}(\Ga_+,\Ga_-) \in \Q$
and $\#^\vir\ovl{\calM}_{g,0,A}^{\R \times Y_\pm,J_\pm}(\Ga_+,\Ga_-)/\R \in \Q$ whenever these moduli spaces have expected dimension zero.
These counts should satisfy various relations which mirror the boundary strata of expected dimension zero for the analogous moduli spaces of expected dimension one.
\end{problem}
\NI Note that these counts must in general be rational numbers, because our moduli spaces are generally at best orbifolds due to the action of biholomorphic raparametrizations for multiple covers.
Also, the formulation in Problem~\ref{prob:SFT_trans} does not cover the full expected functoriality package for SFT, which should also incorporate things like the parametrized moduli spaces mentioned in \S\ref{sec:compactness}, and possibly also moduli spaces of punctured curves satisfying additional geometric constraints (c.f. \S\ref{subsec:geom_constr}), and so on.

Of course, even if we manage to satisfactorily solve Problem~\ref{prob:SFT_trans} and its extensions, one might wonder how we could ever compute anything, especially if the ``curves'' we end up counting are no longer geometrically meaningful objects.
Indeed, even without any extra perturbations, SFT moduli spaces are notoriously difficult to compute.
Here let us briefly mention a few techniques in this direction which make the  problem of computations more tractable than it might at first glance appear.
Firstly, it is sometimes the case that all relevant curves vanish a priori for degree reasons.
For instance, if the contact form $\al$ on $Y$ is such that all Reeb orbits have odd Conley--Zehnder index, then one can check using \eqref{eq:ind} that there are no moduli spaces of the form $\ovl{\calM}_{g,0,A}^{\R \times Y_\pm,J_\pm}(\Ga_+,\Ga_-)/\R$ having expected dimension zero. For example, this is what happens for the exotic Brieskorn contact structures studied in \cite{ustilovsky1999infinitely}.

Secondly, a nice perturbation framework should ideally satisfy the following axiom\footnote{This is sometimes referred to as the ``Obamacare axiom'', the slogan being ``If you like your curve you can keep it.''} 
\begin{axiom}
If an uncompactified SFT moduli space of expected dimension zero is regular and coincides with its SFT compactification, then its virtual count agrees with its classical signed count.
In particular, if the moduli space in question is empty, then this count is necessarily zero.\footnote{Here is it useful to keep in mind that an empty moduli space is vacuously regular. However, a slightly subtle point is that the SFT compactification could be nonempty even if the corresponding uncompactified moduli space is empty.}
\end{axiom}
\NI This axiom is very useful for computations, since in practice many relevant moduli spaces are either regular for a generic choice of almost complex structure (for example if we can rule out multiple covers), or else necessarily empty for elementary reasons (e.g. index considerations, sign considerations, nonnegativity of energy, homological constraints, etc).
In favorable scenarios, one may then be able to explicitly enumerate the regular moduli spaces using say a fibration structure, by reduction to algebraic geometry, using tropical curve counting, etc.

\sss

The SFT transversality problem has inspired a great deal of work in the last several decades, with a number of different projects of varying scopes and degrees of completion.
Although the inner details of these approaches lie beyond the scope of this note, let us mention just a few\footnote{Again, we emphasize that this biased list is by no means exhaustive.} important contributions:

\begin{itemize}
	\item The oldest and best known approach to SFT transversality is the polyfold project of Hofer--Wysocki--Zehnder \cite{hofer2006general,fish2018lectures} (see also the textbook \cite{polyfold_and_fredholm_theory}). In contrast to other approaches based on finite dimensional reduction, the polyfold approach is infinite-dimensional in nature and based on a new paradigm for Fredholm theory.

	\item The implicit atlas formalism of Pardon \cite{Pardon_algebraic_approach} is successfully applied in \cite{Pardcnct} to construct the contact homology algebra for a general contact manifold. This approach is based on topological rather than smooth moduli spaces, and uses a slightly smaller compactification than the usual one discussed in \S\ref{sec:compactness}. At the time of writing, it is not yet understood how to adapt this technique to the setting of linearized contact homology, due to subtleties related to homotopies induced by parametrized moduli spaces. 

	\item Hutchings--Nelson \cite{HuN,HN_hypertight} have been developing an approach to contact homology for three-dimensional contact manifolds, for which the automatic transversality results of \cite{Wendl_aut} can be applied.

	\item Bao--Honda \cite{bao2023semi} gave a construction of the contact homology algebra of a contact manifold based on a notion of semi-global Kuranishi charts.

	\item Ishikawa \cite{ishikawa2018construction} has recently announced a general construction of SFT based on the theory of Kuranishi atlases developed by Fukaya--Ono \cite{fukaya1999arnold}.

\end{itemize}


There are also a number of other approaches to transversality which have been applied in various settings in symplectic geometry and gauge theory; see e.g. \cite[Rmk. 0.2]{Pardcnct} for a comprehensive list of references.
Let us also add the Donaldson divisor approach of Cieliebak--Mohnke \cite{CM1}, which is most effective in closed symplectic manifolds but has been successfully applied in neck-stretching contexts in \cite{CM2} (see \S\ref{sec:applications}), and also the promising recent approach of global Kuranishi charts \cite{abouzaid2021complex,hirschi2022global,hirschi2022global}.

Lastly, let us point out a few more approaches to SFT transversality which are more oblique, in a sense circumventing the issue altogether.

\begin{itemize}

\item It is known that some of the linearized invariants in SFT are closely analogous or even equivalent to known invariants in Floer theory, for which Hamiltonian perturbations typically suffice to achieve transversality. 
For instance, an equivalence between linear contact homology and $S^1$-equivariant symplectic cohomology is presented in \cite{bourgeois2009exact,Bourgeois-Oancea_equivariant}, and the latter (which is rigorously defined in great generality) is used as an effective ersatz for linearized contact homology in e.g. \cite{gutt_pos_equiv,Gutt-Hu}.
Furthermore, \cite{Ekholm-Oancea_DGAS} shows that this equivalence further extends to the contact homology algebra and a CDGA structure defined using symplectic cohomology. We elaborate in the connections between SFT and Floer theory in \S\ref{subsec:ext:Floer} below.

\item Another fruitful approach is to work directly with those SFT moduli spaces which are relevant for a given application, rather than attempting to fully construct coherent algebraic structures.
The idea is that in any given situation there may be only certain moduli spaces which carry important geometric content, and achieving transversality for other moduli spaces may be unnecessary.
Versions of this perspective are applied in qualitative settings in e.g. \cite{ganatra2020embedding}
and in quantitative settings (sometimes under the name ``elementary capacities'' or ``elementary spectral invariants'') in e.g. \cite{mcduff2021symplectic,hutchings2022elementary,edtmair2022elementary,hutchings2022elementaryspectral,chaidez2023elementary}.

\item The version of contact homology implemented in \cite{eliashberg2006geometry} for subdomains of $\R^{2n} \times S^1$ avoids the issue of multiply covered cylinders by restricting to asymptotic Reeb orbits which wind only once around the $S^1$ factor (see \S\ref{subsec:ctct_domains}).

\item Embedded contact homology (see \S\ref{subsec:ECH}) is an analogue of symplectic field theory for three-dimensional contact manifolds which is defined using asymptotically cylindrical punctured curves in their symplectizations, and which rigorously achieves transversality by roughly considering only embedded pseudoholomorphic curves, viewed as currents.
\end{itemize}

\section{Algebraic formalism}\label{sec:formalism}

We are now ready to discuss the algebraic formalism of SFT. 
We first briefly sketch a simplified version of the Eliashberg--Givental--Hofer framework from \cite{EGH2000} in \S\ref{subsec:EGH}.
In \S\ref{subsec:q-only} we discuss a slightly different perspective which is for some purposes easier to conceptualize.
Finally, in \S\ref{subsec:linearize} we discuss an important process called linearization which allows us to define various simplified invariants of symplectic manifolds with contact boundary.

\begin{disclaimer}
While some limited pieces of the SFT package discussed in this section have been constructed rigorously in generality, most of it is continginent on the existence of suitable virtual counts as in \S\ref{sec:transversality}.
We will mostly focus here on the rich algebraic structure of the invariants arising from SFT, with an agnostic approach as to which transversality scheme is used. The same also holds for the various extensions outlined in \S\ref{sec:further}.
\end{disclaimer}

\subsection{Contact homology, rational symplectic field theory, and full symplectic field theory}\label{subsec:EGH}

Fix a Liouville cobordism between contact manifolds $Y_+$ and $Y_-$ which are endowed with nondegenerate contact forms $\al_+$ and $\al_-$ respectively.
Fix generic SFT admissible almost complex structures $J_{\wh{X}}$ on $\wh{X}$ and $J_\pm$ on $\R \times Y_\pm$.
Recall that we seek to define algebraic invariants of $Y_\pm$, as well as morphisms, induced by $X$, from the invariants of $Y_+$ to those of $Y_-$.

By analogy with Morse and Floer homology, a first naive attempt is to define a contact invariant called ``cylindrical contact homology'' as a chain complex $C_\cyl(Y_\pm)$ generated by the Reeb orbits of $\al_\pm$, whose differential counts index one $J_\pm$-holomorphic cylinders in the symplectization $\R \times Y$ (modulo target translations), and a chain map $C_\cyl(Y_+) \ra C_\cyl(Y_-)$ which counts index zero $J_{\wh{X}}$-holomorphic cylinders in $\wh{X}$.
Here ``count'' should be interpreted in a virtual sense as in Problem~\ref{prob:SFT_trans}, and hence these may implicitly depend on some auxiliary abstract perturbation data (as well as our choices of contact forms and almost complex structures).

There are a few issues with this approach. One relatively minor point is that in order to define suitable signed counts we need to coherently orient our moduli spaces, and for this we must restrict to Reeb orbits which are \hl{good}. Here we say that a Reeb orbit is \hl{bad} if it is an even-fold cover of another Reeb orbit of opposite parity, otherwise it is good. When more care is taken with asymptotic markers, we see that bad Reeb orbits appear with an even number of choices of asymptotic marker and these come in cancelling pairs, so that we can and should ignore them.

A bigger issue is that this purported differential does not always square to zero. Naively, the differential would square to zero if we could show that any index two cylinder in $\R \times Y_\pm$ can only break into a two-level building, with each level consisting of an index one  cylinder in $\R \times Y_\pm$. However, a priori the SFT compactness theorem allows other possible degenerations of a cylinder, for instance into a pair of pants in an upper level and a cylinder and plane in a lower level (see Figure~\ref{fig:bad_breaking}).
Note that the analogous picture involving a plane in the upper level cannot occur, because by Stokes' theorem any asymptotically cylindrical curve in a symplectization must have at least one positive end.

	\begin{figure}
	\centering
  \includegraphics[scale=.6]{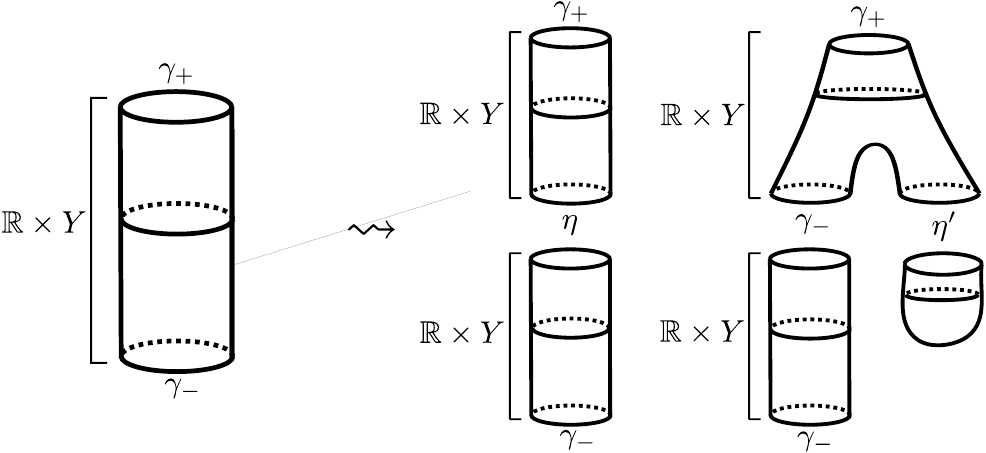}
  \caption{Two possible breakings of a cylinder in a symplectization into a two level building. The latter breaking prevents the naive differential which counts cylinders from squaring to zero, and this motivates the definition of the contact homology algebra.}
  \label{fig:bad_breaking} 	
	\end{figure}

An elegant resolution is to simply take into account curves with extra negative ends from the outset, defining an algebraic structure based on genus zero punctured curves with one positive end and any number (possibly zero) of negative ends, as in the left panel of Figure~\ref{fig:SFT_schematic}.
Given a contact manifold $Y^{2n-1}$ with nondegenerate contact form $\al$, let $\calP_Y$ denote the set of good Reeb orbits\footnote{More precisely, these are unparametrized Reeb orbits, but we remember the covering multiplicity.} in $Y$, and let $V_Y := \Q\lan q_\ga\;|\; \ga \in \calP_Y \ran$ be the graded rational vector space with a basis element $q_\ga$ for each good Reeb orbit $\ga$ of $Y$, with grading $|q_{\ga}| = \cz(\ga) + n-3$.\footnote{Here we interpret these gradings in $\Z/2$, noting that the Conley--Zehnder index has a well-defined parity independent of any choice of trivialization of the contact hyperplanes along it. Under additional topological assumptions this can be upgraded to a $\Z$-grading. The grading convention $|q_\ga| = \cz(\ga) + n-3$ is sometimes called the \hl{SFT index}.}
Given an SFT admissible almost complex structure $J$ on $\R \times Y$, we define a commutative differential graded algebra $C_\cha(Y)$ over $\Q$ as follows.

	\begin{figure}
	\centering
  \includegraphics[scale=.7]{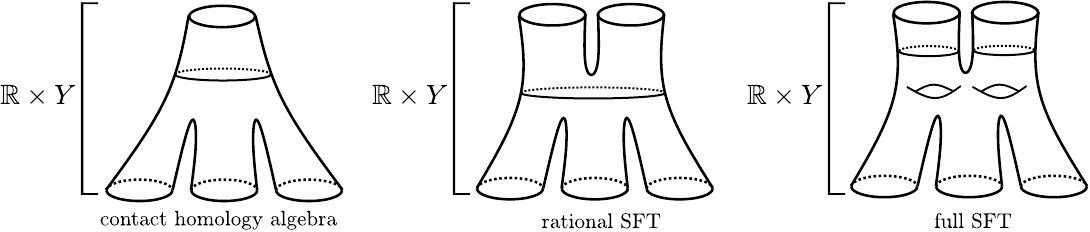}
  \caption{A schematic picture of the curves involved in the contact homology algebra, rational symplectic field theory, and full symplectic field theory.}
  \label{fig:SFT_schematic} 
	\end{figure}

\begin{itemize}

	\item As a graded commutative algebra, $C_\cha(Y)$ is the free graded commutative algebra $\calA_Y := \sym(V_Y) = \Q[q_\ga\;|\; \ga \in \calP_Y]$ with a formal variable $q_\ga$ for each good Reeb orbit $\ga$ of $Y$. 
	Here graded commutativity means that we have $q_{\ga_1}\cdots q_{\ga_{i}}q_{\ga_{i+1}}\cdots q_{\ga_k} = (-1)^{|q_{\ga_i}| |q_{\ga_{i+1}}|}q_{\ga_1}\cdots q_{\ga_{i+1}}q_{\ga_{i}}\cdots \ga_k$, and in particular $q_{\ga} q_{\ga} = 0$ if $|q_{\ga}|$ is odd.

	\item For $\ga \in \calP_Y$, the differential $\bdy_{\cha}(q_\ga)$ is given by
		\[\bdy_{\cha} (q_{\ga}) := \sum\limits_{A,\Ga_-}\tfrac{1}{\comb(\Ga_+,\Ga_-)} \cdot \#^\vir \ovl{\calM}_{0,0,A}^{\R \times Y,J}(\Ga_+;\Ga_-)/\R \cdot q_{\ga_1}\cdots q_{\ga_k},\] 
		where we put $\Ga_+ = (\ga)$, the sum is over tuples $\Ga_- = (\ga_1,\dots,\ga_k)$ of good Reeb orbits and homology classes $A$, and our convention is that $\#^\vir \ovl{\calM}_{0,0,A}^{\R \times Y,J}(\Ga_+;\Ga_-)/\R = 0$ unless the expected dimension is zero. Here $\comb(\Ga_+,\Ga_-) \in \Z_{\geq 1}$ is a combinatorial factor related to the ordering and covering multiplicities of the Reeb orbits in $\Ga_+,\Ga_-$, which we will mostly gloss over here (although it is necessary to get the correct gluing factors).\footnote{More precisely, we have $\comb(\Ga_+,\Ga_-) = \mu_{\Ga_-} \ka_{\Ga_-}$, where $\mu_{\Ga_-}$ is the number of permutation symmetries of the tuple $\Ga_-$, and $\ka_{\Ga_-}$ is the products of the covering multiplicies of the Reeb orbits in $\Ga_-$. As it happens there are a few other possible conventions in SFT which give rise to slightly different combinatorial factors.}
		The differential $\bdy_\cha$ is extended to all of $\calA_Y$ by the (graded) Leibniz rule.
\end{itemize}
The homology of $C_\cha(Y)$ is a graded commutative algebra called the \hl{contact homology algebra} of $Y$, which (assuming a suitable solution to the transversality problem as in \S\ref{sec:transversality}) depends only on the contactomorphism type of $Y$.

Counting similar types of curves in the completed symplectic cobordism $\wh{X}$ induces a DGA homomorphism $\Phi_\cha: C_{\cha}(Y_+) \ra C_{\cha}(Y_-)$ (maps like this induced by symplectic cobordisms are often called \hl{cobordism maps}).
More precisely, for $\ga \in \calP_{Y_+}$ we put
\begin{align*}
\Phi_\cha(q_\ga) := \sum\limits_{A,\Ga_-} \tfrac{1}{\comb(\Ga_+,\Ga_-)} \cdot \#^\vir \ovl{\calM}_{0,0,A}^{\wh{X},J_{\wh{X}}}(\Ga_+;\Ga_-) \cdot q_{\ga_1}\cdots q_{\ga_k},
\end{align*}
with $\Ga_+ = (\ga)$.
This extends to all of $\calA_{Y_+}$ by multiplicativity, or equivalently we can think of $\Phi_X$  as counting disconnected curves in $\wh{X}$ such that each component is rational with one positive end and many negative ends.


\sss

Next, we seek to incorporate all rational curves with any number of positive and negative ends. 
Since naively gluing two rational curves tends to produce a curve of higher genus, some care is needed to formulate the correct algebraic structure.
At this point it is useful to package together all counts of index one rational punctured curves in the symplectization $\R \times Y$ into a single generating function.
Let \[\frakB_Y := \calA_Y \ll p_\ga \;|\; \ga \in \calP_Y \rr\] denote the graded commutative algebra of formal power series in variables $p_\ga$ with $|p_\ga| = -\cz(\ga) + n-3$ for each good Reeb orbit, with coefficients in $\calA_Y = \Q[ q_\ga\;|\; \ga \in \calP_Y]$.
The \hl{RSFT Hamiltonian} is defined by
\begin{align*}
\hh_Y := \sum\limits_{\Ga_+,\Ga_-,A} \tfrac{1}{\op{comb}(\Ga_+,\Ga_-)} \cdot \#^\vir\ovl{\calM}_{0,0,A}^{\R \times Y,J}(\Ga_+;\Ga_-)/\R \cdot  p_{\ga_1^+}\cdots p_{\ga_{s_+}^+} q_{\ga^-_1}\cdots q_{\ga^-_{s_-}} \in \frakB_Y,
\end{align*}
where the sum is over all collections of good Reeb orbits $\Ga_\pm = (\ga_1^\pm,\dots,\ga_{s_\pm}^\pm)$ and homology classes $A$.

We use $\hh_Y$ to define a differential on $\frakB_Y$ as follows.
First, we equip $\frakB_Y$ with the Poisson bracket $\{-,-\}$ given by
\begin{align*}
\{f,g\} := \sum\limits_{\ga \in \calP_Y} \ka_\ga \left(\tfrac{\bdy f}{\bdy p_\ga}\tfrac{\bdy g}{\bdy q_\ga}- (-1)^{\deg(f) \deg(g)} \tfrac{\bdy g}{\bdy p_\ga}\tfrac{\bdy f}{\bdy q_\ga} \right)
\end{align*}
for any monomials $f,g \in \frakB_Y$, where $\ka_\ga$ denotes the covering multiplicity of the Reeb orbit $\ga$.
This turns $\frakB$ into a graded Poisson algebra.
It turns out that the curve counting relations carried by the boundaries of moduli spaces of index two rational punctured curves in the symplectization $\R \times Y$ can all be succinctly encoded into a single equation, the \hl{RSFT Hamiltonian master equation}:
\begin{align}\label{eq:RSFT_master}
\{ \hh_Y,\hh_Y\} = 0.
\end{align}
It then follows that the differential 
$\bdy_{\rsft} := \{\hh_Y,-\}$ on $\frakB_Y$ satisfies $\bdy_{\rsft}^2 = 0$, and it makes $\frakB_Y$ into a differential graded Poisson algebra, which we will denote by $C_{\rsft}(Y)$.
In particular, the homology of $C_{\rsft}(Y)$ is a graded Poisson algebra and a contact invariant of $Y$, which we will call the \hl{rational symplectic field theory} of $Y$.

Similarly, we can package all index zero rational $J_{\wh{X}}$-holomorphic curves in $\wh{X}$ into a generating function $\ff_{\wh{X}}$ called the \hl{RSFT potential} of $\wh{X}$, which is a power series in variables $p_\ga$ for $\ga \in \calP_{Y_+}$ and $q_\eta$ for $\eta \in \calP_{Y_-}$. 
The relations given by analyzing boundaries of index one rational curves in $\wh{X}$ translates into a single \hl{RSFT potential master equation} relating $\ff_{\wh{X}},\hh_{Y_+},\hh_{Y_-}$. 
Instead of a cobordism map, one can view $\ff_{\wh{X}}$ as producing a \hl{Lagrangian correspondence} which transforms $\hh_{Y_+}$ and $\hh_{Y_-}$ into each other (see \cite[\S2.3.2]{EGH2000}).

\sss

Finally, we incorporate curves of arbitrary genus.
In order to write down a generating function for all index one punctured curves in the symplectization $\R \times Y$, since the SFT compactness theorem requires an a priori bound on the genus, we must incorporate an additional formal variable $\hbar$.
Let $\frakW_Y$ be the graded associative algebra over $\Q$ with generators $q_\ga,p_\ga$ for $\ga \in \calP_Y$ (with the same gradings as before) and $\hbar$ with $|\hbar| = 2(n-3)$, subject to the relations that all generators graded commute except for 
\begin{align*}
[p_\ga,q_\ga] := p_\ga \star q_\ga - (-1)^{|p_\ga| |q_\ga|} q_\ga \star p_\ga = \ka_\ga \hbar
\end{align*}
(here $\star$ denotes the product on $\frakW_Y$).
We note that $\frakW$ is an example of a \hl{graded Weyl algebra}, because it can be faithfully represented as an algebra of formal differential operators acting on $\calA_Y[\hbar]$ on the left via the substitution $p_\ga \mapsto \ka_\ga \hbar \tfrac{\partial}{\partial q_\ga}$.
The \hl{full SFT Hamiltonian} is now defined by
\begin{align*}
\HH_Y := \sum\limits_{\Ga_+,\Ga_-,g,A} \tfrac{1}{\op{comb}(\Ga_+,\Ga_-)} \cdot \hbar^{g-1}\cdot \#^\vir\ovl{\calM}_{g,0,A}^{\R \times Y,J}(\Ga_+;\Ga_-)/\R \cdot  p_{\ga_1^+}\cdots p_{\ga_{s_+}^+} q_{\ga^-_1}\cdots q_{\ga^-_{s_-}} \in \tfrac{1}{\hbar}\frakW_Y.
\end{align*}

The relations induced by boundaries of index $2$ moduli spaces of punctured curves in $\R \times Y$ can now be encoded in a single \hl{full SFT Hamiltonian master equation}:
\begin{align}\label{eq:SFT_master}
\HH_Y \star \HH_Y = 0.
\end{align}
Note that this extends \eqref{eq:RSFT_master} in the sense 
$[\HH_Y,\HH_Y] = \tfrac{1}{\hbar}\{\hh_Y,\hh_Y \} + h.o.t.$.
In the quantum mechanical language of \cite{EGH2000}, $C_{\rsft}(Y)$ is the \hl{semi-classical approximation} of $C_{\sft}(Y)$, and $C_{\cha}(Y)$ is its classical approximation.

Since $\HH_Y$ is odd, we can equivalently write ~\eqref{eq:SFT_master} as $[\HH_Y,\HH_Y] = 0$,
where the graded commutator of homogeneous elements $F,G$ is defined by
$[F,G] := F \star G - (-1)^{\deg(F) \deg(G)} G \star F$.
It follows that the differential $\bdy_{\sft} := [\HH_Y,-]$ on $\frakW_Y$ satisfies $\bdy^2_{\sft} = 0$ and is a derivation with respect to the product $\star $. This makes $\frakW_Y$ into a differential graded algebra, which we denote by $C_{\sft}(Y)$.
In particular, the homology of $C_{\sft}(Y)$ is a graded associative algebra and a contact invariant of $Y$, which we call the \hl{symplectic field theory} of $Y$.
Similarly, the generating function of punctured curves of arbitrary genus in the cobordism $\wh{X}$ leads to rise to the \hl{full SFT potential} $\FF_{\wh{X}}$, which is related to $\HH_{Y_+}$ and $\HH_{Y_-}$ by the \hl{full SFT potential master equation}.

\subsection{Reformulation with only $q$ variables}\label{subsec:q-only}

	There are other ways of packaging the above curve counts into algebraic structures, which can sometimes be more convenient depending on the intended applications (see e.g. \cite{latschev2011algebraic,HSC,ganatra2020embedding,moreno2020landscape}).
	Notice that, in the above formulation, $C_\cha(Y)$ involves only the variables $q_\ga$, whereas $C_\rsft(Y)$ and $C_\sft(Y)$ require also the variables $p_\ga$.
	We now mention reformulations of these latter invariants without the $p_\ga$ variables, with the virtue that we get cobordism maps closely analogous to what we have for $C_\cha(Y)$.

	Recall that the graded Weyl algebra $\frakW_Y$ can be represented by formal differential operators acting on $\calA_Y\ll \hbar \rr$, where $\calA_Y = \Q[q_\ga\;|\; \ga \in \calP_Y]$.
	In particular, under this representation the full SFT Hamiltonian $\HH_Y$ corresponds to a map $\bdy_\sft^\qonly: \calA_Y\ll \hbar \rr \ra \calA_Y \ll \hbar \rr$, and the master equation $\HH_Y \star \HH_Y = 0$ translates into $(\bdy_\sft^\qonly)^2 = 0$.
	This makes $\calA_Y\ll \hbar \rr$ into a chain complex, which we denote by $C_\sft^\qonly(Y)$.
	The homology of $C_\sft^\qonly(Y)$ is a contact invariant of $Y$ which gives a different formulation of its full SFT.
	Note that although $C_\sft^\qonly(Y)$ is smaller than $C_\sft(Y)$ as an algebra, the differential $\bdy_\sft^\qonly$ does not satisfy a Leibniz rule, and hence the homology of $C_\sft^\qonly(Y)$ does not inherit a product.
	Rather, $\bdy_\sft^\qonly$ decomposes into a sum of differential operators of increasing orders, making $C_\sft^\qonly(Y)$ into a $\op{BV}_\infty$ algebra in the language of \cite[\S5]{cieliebak2009role}.
	Furthermore, one can show that a Liouville cobordism $X$ between $Y_+$ and $Y_-$ induces a $\op{BV}_\infty$ morphism $\Phi_\sft^\qonly: C_\sft^\qonly(Y_+) \ra C_\sft^\qonly(Y_-)$, which in particular is a chain map.

	It is also possible to reformulate rational symplectic field theory using only $q$ variables, although some extra care is needed to make sure we only glue two rational curves along a single pair of punctures.
	An algebraic description of RSFT with only $q$ variables as a chain complex was sketched in \cite{hutchings_2013}, and worked out in detail in \cite[\S3.4]{HSC} using the language of $\Li$ algebras. 
	Namely, we can view index one rational punctured curves in $\R \times Y$ as defining an $\Li$ algebra whose underlying chain complex is $C_\cha(Y)$. This means that we have graded symmetric operations $\otimes^{k} C_\cha(Y) \ra C_\cha(Y)$ for all $k \in \Z_{\geq 1}$ which satisfy the $\Li$ structure equations (an infinite sequence of quadratic relations). In particular, the \hl{bar complex} of this $\Li$ algebra is a chain complex $C_\rsft^\qonly(Y)$ whose underlying vector space is $\sym(\calA_Y)$, the double symmetric tensor algebra on $V_Y$.\footnote{Strictly speaking it is more natural to take the reduced bar complex -- see e.g. \cite[Def. 2.4]{HSC} for more details.}
	Moreover, the Liouville cobordism $X$ induces an $\Li$ homomorphism from $C_\cha(Y_+)$ to $C_\cha(Y_-)$, and in particular a chain map $\Phi_\rsft^\qonly: C_\rsft^\qonly(Y_+) \ra C_\rsft^\qonly(Y_-)$.
	A further refinement of this structure which takes into account the algebra structure on $\calA_Y$ is also described in \cite{moreno2020landscape} using the language of ``bi-Lie algebras'', and a detailed comparison between these $q$ variable only approaches to RSFT and the original formalism of Eliashberg--Givental--Hofer appears in \cite{latschev2022remarks}.

\subsection{Linearization}\label{subsec:linearize}

Observe that the algebra $\calA_Y$ has an increasing filtration by word length, but unfortunately the differential $\bdy_\cha$ is not in general nondecreasing with respect to this filtration, essentially due to the possibility of index $1$ planes in the symplectization $\R \times Y$.
In the absence of such planes, or more precisely when $\#^\vir \ovl{\calM}^{\R \times Y,J}_{0,0,A}((\ga);\nil)/\R = 0$ for all $\ga \in \calP_Y$, then $\bdy_\cha$ does preserve this word length filtration, and in this case we will say that $C_\cha(Y)$ is {\em trivially augmented}.
If $C_\cha(Y)$ is trivially augmented, then by restricting and projecting $\bdy_\cha$ to the subspace of words of length one, we get a differential $V_\calP \ra V_\calP$ which squares to zero.
In particular, this gives a chain complex which is much smaller than $C_\cha(Y)$, but it is not a priori a contact invariant of $Y$, 
since e.g. $C_\cha(Y)$ might not be trivially augmented for a different choice of contact form or almost complex structure.

In fact, given a unital DGA morphism $\aug: C_\cha(Y) \ra \Q$ (also known as an \hl{augmentation}), we can modify the CDGA $C_\cha(Y)$ so that it becomes trivially augmented.
Namely, let $F^\aug: \calA_Y \ra \calA_Y$ be the algebra isomorphism defined on generators by $F^\aug(q_\ga) = q_\ga + \aug(q_\ga)$, and define a new differential $\bdy_\cha^\aug$ on $\calA_Y$ by 
$\bdy_\cha^\aug := (F^\aug) \circ \bdy_\cha \circ (F^\aug)^{-1}$.
Noting that $(F^\aug)^{-1}(q_\ga) = q_\ga - \aug(q_\ga)$, we have $\bdy_\cha^\aug(q_\ga) = F^\aug \bdy (q_\ga)$, whose word length zero piece is $\aug(\bdy_\cha(q_\ga))$, and this vanishes since $\aug$ is an augmentation.
It follows that $\bdy_\cha^\aug$ is nondecreasing with respect to the word length filtration on $\calA_Y$, so we get an induced differential on the subspace of words of length one (i.e. $V_Y$), which we denote by $\bdy_\chlin$.
We denote the corresponding chain complex by $C_\chlin(\aug)$, and we refer to its homology as the \hl{linearized contact homology} of $Y$ with respect to the augmentation $\aug$.
The set of all linearized contact homologies over all augmentations of $C_\cha(Y)$ is expected to be a contact invariant of $Y$.\footnote{While the analogous invariant for Legendrian knots in $\R^3$ has been implemented to great effect in e.g. \cite{Chekanov_DGA}, to our knowledge the set of all linearized contact homologies for a closed contact manifold has not yet been implemented in full detail, due to some technical subtleties related to DGA homotopies (c.f. \cite[\S1.8]{Pardcnct}).}

Note that we can also consider the full homology of $\calA_Y$ with respect to the twisted differential $\bdy_\cha^\aug$. In fact, as an algebra this is just isomorphic to the usual contact homology algebra of $X$, since $\bdy_\cha^\aug$ is conjugate to $\bdy_\cha$, but nevertheless it is a somewhat nicer algebraic object since it carries a word length filtration.
We will denote the CDGA $(\calA_Y,\bdy_\cha^\aug)$ by $C_\chalin(X)$.


\sss

Now suppose that $X$ is a Liouville domain with contact boundary $Y$, i.e. $X$ is a Liouville cobordism with positive contact boundary $Y$ and empty negative boundary.
In this case, the cobordism map induced by $X$ is precisely an augmentation $\aug_X: C_\cha(Y) \ra \Q$.
This gives rise to a linearized chain complex $C_\chlin(X) := C_\chlin(\aug_X)$ whose corresponding homology is a symplectic invariant of $X$.
Moreover, the differential $\bdy_\chlin$ on $C_\chlin$ has a more appealing geometric description (at least heuristically) as 
a count of two level pseudoholomorphic buildings, where the top level is an index $1$ rational curve in $\R \times Y$ with one positive end, and the bottom level is a collection of index $0$ planes in $\wh{X}$, such that all but one of the negative ends of the upper level curve is matched with a plane in the lower level (see the left panel of Figure~\ref{fig:SFT_schematic_linearized}). It is useful to think such a configuration as a cylinder in $\R \times Y$ with extra negative ends capped by planes in $\wh{X}$ (these extra capped ends are called \hl{anchors} in \cite{bourgeois2012effect}).

\begin{rmk}
Sometimes we can rule out index one planes in $\R \times Y$ for degree or fundamental group reasons.
Then $C_\cha(Y)$ is already trivially augmented, i.e. $\aug_X(q_\ga) = 0$ for all $\ga \in \calP_Y$, and we have $C_\cha(Y) = C_\chalin(X)$.
In this case it is natural to think of $C_\chlin(X)$ as an effective stand in for $C_\cyl(Y)$.
\end{rmk}
\begin{rmk}
It can also be the case that $C_\cha(Y)$ has no augmentations whatsoever.
For example, this holds whenever the homology of $C_\cha(Y)$ is trivial, since then the empty word is a boundary and hence must be sent to both $1$ and $0$ in $\Q$ under any augmentation, which is impossible. 
Notably, this holds when the contact manifold $Y$ is overtwisted, and it follows that such $Y$ cannot admit any Liouville filling.
\end{rmk}

Similarly, we can use the Liouville filling $X$ of $Y$ (or more generally any abstract augmentation, suitably defined) to linearize the rational and full symplectic field theory of $X$, giving rise to symplectic invariants of $X$ which are potentially more tractable than the RSFT and full SFT of $Y$.
In essence, the augmentation induces a change of coordinates, after which our invariant becomes trivially augmented in the sense that there are no contributions from index $1$ curves with no negative ends in $\R \times Y$.
This results in somewhat nicer chain complexes with simpler differentials, and it also allows us to define intermediate invariants with simplified algebraic structures.
For example, the chain-level invariant $C_\sftlin^\qonly(Y)$ obtained by twisting the differential on $C_\sft^\qonly(X)$ by the augmentation induced by $X$ is a special type of $\op{BV}_\infty$ which is called an $\op{IBL}_\infty$ algebra in \cite{cieliebak2020homological}.
In the rational case, after twisting the differential of $C_\rsft^\qonly(Y)$ to obtain $C_\rsftlin^\qonly(X)$, there is a self-consistent substructure which counts rational curves in $\R \times Y$ with one negative end and many positive ends, plus additional anchors in $\wh{X}$. This structure
can be viewed as an $\Li$ algebra whose underlying chain complex is $C_\chlin(X)$, and in particular its bar complex $C_{\mathcal{B}\chlin}(X)$ is a chain complex with underlying vector space $\calA_Y$ (see \cite[\S3.4.3]{HSC}).
Note that, modulo the anchors, $C_{\mathcal{B}\chlin}(X)$ is an ``upside down'' version of the contact homology algebra $C_{\cha}(X)$; see Figure~\ref{fig:SFT_schematic_linearized} for a schematic diagram of these linearized structures.

\begin{rmk}
For a contact manifold $Y$, one can show that at the homology level the vanishing of the contact homology algebra of $Y$ is equivalent to the vanishing of the rational SFT or full SFT of $Y$, and when these vanish we say that $Y$ is \hl{algebraically overtwisted} (see \cite{bourgeois2010towards}). One potential downside is that the higher parts of SFT do not provide any additional information for distinguishing algebraically overtwisted contact manifolds from genuinely overtwisted ones.
A similar phenomenon appears in \cite[App. B]{ng2010rational} in the relative setting, where RSFT does not provide any interesting information for stabilized Legendrian knots.
\end{rmk}

	\begin{figure}
	\centering
  \includegraphics[scale=.75]{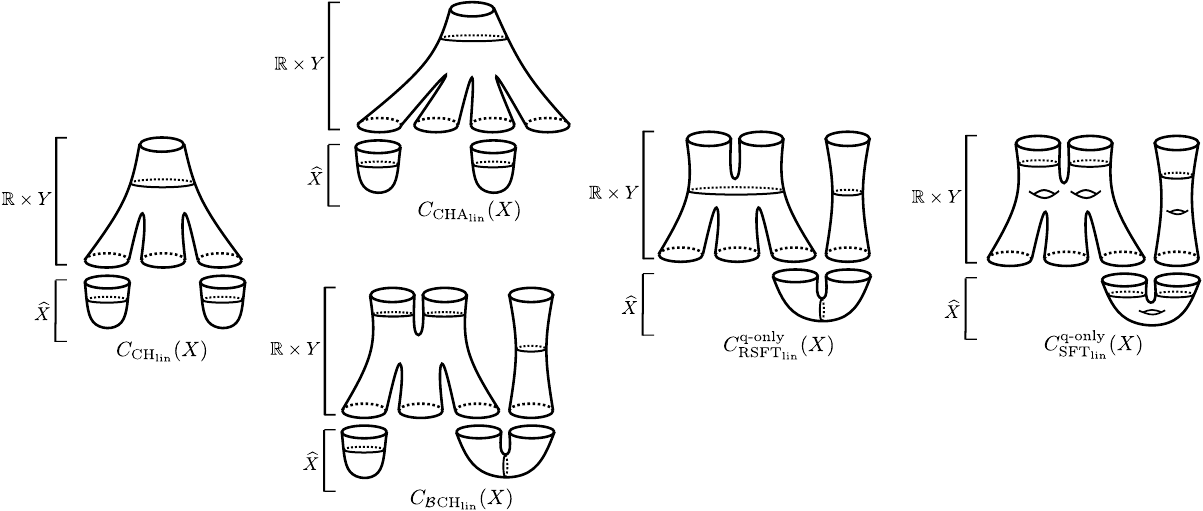}
  \caption{A schematic picture of the anchored curves involved after linearized the invariants of the contact manifold $Y$ induced by its Liouville filling $X$. Note that the invariants $C_\chlin(X)$ and $C_{\mathcal{B}\chlin}(X)$ have no natural unlinearized counterparts, whereas the remaining three invariants only differ by their unlinearized versions by a change of coordinates.}
  \label{fig:SFT_schematic_linearized} 
	\end{figure}

\section{Applications}\label{sec:applications}


There are many important applications of SFT in the literature, and undoubtedly plenty more yet to be discovered.
Noteworthy initial proofs of concept include distinguishing Legendrian knots \cite{Chekanov_DGA} and contact spheres \cite{ustilovsky1999infinitely}, nonfillability of overtwisted contact $3$-manifolds \cite{eliashberg_contact,MLYau_vanishing}, new recursive formulas for Gromov--Witten invariants \cite[\S2.9.3]{EGH2000}, and so on.
Incidentally, most of these early applications involve only rational curves with one positive end, but see e.g. \cite{latschev2011algebraic} for an application to obstructing symplectic cobordisms which relies on higher genus curves.

In this section we will content ourselves with a simple but beautiful argument which uses SFT to restrict the topology of Lagrangian submanifolds, following \cite[\S1.7]{EGH2000}.
This argument does not make use of the algebraic formalism discussed in \S\ref{sec:formalism}, but it does rely in an essential way on the SFT compactness theorem, and it also highlights the relevance of transversality.

\begin{thm}[\cite{viterbo1994properties}]\label{thm:hyp} 
Let $M$ be a smooth complex projective variety of complex dimension $n \geq 3$ which is uniruled, equipped with its K\"ahler symplectic form.
Let $L \subset M$ be a closed embedded Lagrangian submanifold. Then $L$ does not admit any Riemannian metric with negative sectional curvature.
\end{thm}

\NI Here the uniruled condition means that there is a rational curve through every point in an open dense subset of $M$, and this holds whenever $M$ is Fano (see \cite{miyaoka1986numerical,kollar2013rational}). We could also replace this with the assumption that $M$ has a nonvanishing Gromov--Witten invariant with one point constraint.

The relevance of sectional curvature in Theorem~\ref{thm:hyp} is the following.
Let $S^*L \subset T^*L$ denote the unit cosphere bundle with respect to a Riemannian metric $g$ on $L$.
Then the (unparametrized) Reeb orbits in $S^*L$ are in bijective correspondence with oriented closed geodesics in $L$. We will let $\wt{\al}$ denote the Reeb orbit lift to $S^*L$ of a closed oriented geodesic $\al$ in $L$.
If $L$ is orientable, there is a canonical way to define Conley--Zehnder indices for Reeb orbits in $S^*L$, such that $\cz(\wt{\al})$ equals the Morse index of $\al$ and the Chern number term in \eqref{eq:ind} vanishes.
When $g$ has nonpositive sectional curvature, it is a classical fact that all geodesics $\al$ in $L$ are homotopically essential and satisfy $\morse(\al) = 0$ (see e.g. \cite{csirikcci2024morse}).
When $g$ has strictly negative sectional curvature, the geodesics of $L$ are isolated and lift to nondegenerate Reeb orbits in $S^*L$.
Thus for punctured curves in $T^*L$ with asymptotics $\wt{\al}_1,\dots,\wt{\al}_{s}$, we have
\begin{align}\label{eq:ind_hyp}
\ind\,\calM_{g,0,A}^{T^*L,J}((\wt{\al}_1,\dots,\wt{\al}_{s});\nil) = (n-3)(2-2g - s).
\end{align}
Note that we necessarily have $s \geq 2$ since $L$ has no contractible geodesics, and hence the quantity in \eqref{eq:ind_hyp} is nonpositive.

As for the uniruledness assumption in Theorem~\ref{thm:hyp}, according to \cite{kollar2013rational,ruan_virtual} this implies that there exists a homology class $A \in H_2(M)$ such that for any compatible almost complex structure $J$ on $M$ and any point $p \in M$ there is a $J$-holomorphic sphere $u: \CP^1 \ra M$ with $[u] = A$ which passes through $p$. 
This is closely related to nonvanishing of the genus zero Gromov--Witten invariant of $M$ in homology class $A$ with one point constraint, although strictly speaking the latter is defined in terms of stable maps, which could a priori have several components.

\begin{proof}[Proof of Theorem~\ref{thm:hyp}]
Suppose by contradiction that $L$ admits a Riemannian metric with negative sectional curvature.
We will assume that $L$ is orientable (otherwise one can argue in terms of the orientable double cover of $L$).
By Weinstein's Lagrangian neighborhood theorem, there is a neighborhood $U$ of $L$ in $M$ which is symplectomorphic to the $\eps$-disk cotangent bundle $T_\eps^*L$, and after rescaling the metric we may assume $\eps = 1$.

Let $J_1,J_2,J_3,\dots$ be a sequence of compatible almost complex structures on $M$ which realizes neck stretching along $\bdy U \cong S^*L$. Recall that this roughly means that these become cylindrical (and in particular translation invariant) on increasingly large collar neighborhoods of $\bdy U$.
In the limit we arrive at a split symplectic cobordism whose pieces are identified with $T^*L$ and $M \setminus L$, each carrying an SFT admissible almost complex structure. We can assume that the almost complex structure $J_{T^*L}$ on $T^*L$ is chosen generically.

By the discussion preceding the proof, we can fix $A \in H_2(M)$ and generic $p\in M$ such that for each $i \in \Z_{\geq 1}$ there exists a $J_i$-holomorphic sphere $u_i: \CP^1 \ra M$ with $[u_i] = A$ which passes through $p$.
By the SFT compactness theorem, there is a subsequence which converges to a pseudoholomorphic building consisting of a bottom level in $T^*L$, some number of intermediate symplectization levels in $\R \times S^*L$, and a top level in $M \setminus L$, where some component $C$ in the bottom level passes through $p$.

Let $\uvl{C}$ be the underlying simple curve of $C$ (so $\uvl{C} = C$ unless $C$ is a multiple cover).
We will view $\uvl{C}$ as an asymptotically cylindrical $J_{T^*L}$-holomorphic curve in $T^*L$ with an additional marked point in its domain which is required to map to $p$.
By generic transversality for simple curves and our genericity assumption on $J_{T^*L}$, we can assume that $\uvl{C}$ is regular, and in particular has nonnegative index.
 On the other hand, by \eqref{eq:ind_hyp} the index of $\uvl{C}$ is $(n-3)(2-2g-s) - (2n-2) < 0$, where the last term takes into account the point constraint, so this gives a contradiction.

\end{proof}

Evidently Theorem~\ref{thm:hyp} breaks down if we replace negative sectional curvature with nonpositive sectional curvature, due to the existence of Lagrangian tori in complex projective space (e.g. the Clifford torus).
Nevertheless, the following result puts nontrivial restrictions on such Lagrangians. We denote by $\CP^n$ complex projective space with its Fubini-Study symplectic form scaled such that lines have area $\pi$.

\begin{thm}[{\cite[Thms. 1.1 and 1.2]{CM2}}]\label{thm:CM} Let $L \subset \CP^n$ be a closed Lagrangian submanifold admitting a metric of nonpositive sectional curvature. Then there exists a smooth map $f:(D, \partial D) \to (\CP^n, L)$ with $f^* \omega \ge 0$ and $$0 < \int_D f^* \omega \le \frac{\pi}{n+1}.$$
Moreover, if $L$ is orientable and either monotone or a torus, then we can take the disk $f$ to have Maslov index $2$.
\end{thm}

\NI In particular, the second statement in the case when $L$ is a torus verifies a 1988 conjecture of Audin \cite{audin1988fibres} stating that Lagrangian tori in $\C^n$ bound Maslov $2$ disks.

\begin{proof}[Proof sketch of Theorem~\ref{thm:CM}]
	
The proof idea, originally suggested by Yasha Eliashberg, is to extend the neck stretching argument used in the proof of Theorem~\ref{thm:hyp}. Now the point constraint above is replaced with a higher index local tangency constraint.
This means that we consider rational curves in $\CP^n$ which pass through a chosen point $p$ and are tangent to order $m-1$ (i.e. contact order $m$) to a generically chosen local holomorphic divisor $D$ through $p$.\footnote{This constraint is denoted by $\lll \T_D^{(m)}	p \rrr$ in \S\ref{subsec:geom_constr} below.}
This roughly amounts to specifying the $(m-1)$-jet of the curve at a point, thereby imposing a constraint of (real) codimension $2n+2m-4$.
We will restrict to the line class $[\CP^1] \in H_2(\CP^n)$ and put $m=n$, so that we expect a finite number of such curves which are $J$-holomorphic for any generic compatible almost complex structure $J$.
In fact, by \cite[Prop. 3.4]{CM2} the number of such curves is precisely $(n-1)!$, and in particular nonzero.

Now we examine how these curves degenerate as we stretch the neck along the boundary of a small Weinstein neighborhood of $L$ as in the proof of Theorem~\ref{thm:hyp}.
Note that the geodesics of $L$ (and hence also the Reeb orbits of $S^*L$) typically appear in families of dimension at most $n-1$, but we can either work in a Morse--Bott setting or make a small generic perturbation to achieve nondegenerate Reeb dynamics.
At any rate, in the neck stretching limit there must be some pseudoholomorphic building consisting of a bottom level in $T^*L$, some number of intermediate symplectization levels in $\R \times S^*L$, and a top level in $\CP^n \setminus L$, where components in the bottom level carry the local tangency constraint.
As in the proof of Theorem~\ref{thm:hyp}, curves in the bottom level are $J_{T^*L}$-holomorphic, with $J_{T^*L}$ a generic SFT admissible almost complex structure on $T^*L$.
A straightforward calculation shows that the index of a single curve $C$ carrying the local tangency constraint is $2k-2-2n$, where $k$ is the number of positive punctures of $C$.
In particular, if such a $C$ is simple then we can assume that it has nonnegative index, and hence $k \geq n+1$. 
In fact, by passing to the underlying simple curve and inspecting the Riemann--Hurwitz formula, one can show that $k \geq n+1$ holds also in the case when $C$ is a multiple cover. 
In principle the local tangency constraint could also lie on a ghost component, but in this case one can show that the constraint is effectively carried by a union of nonconstant components in the limiting building, and the total number of positive punctures of all bottom level curves must still be at least $n+1$.
Since our pseudoholomorphic building has total genus zero, we can combine its remaining components into $k$ smooth disks $f_1,\dots,f_k: (D,\bdy D) \ra (\CP^n,L)$, each having positive symplectic area. 
Since the sum of their areas is bounded from above by $\pi$, it follows that at least one $f_i$ must have area at most $\tfrac{\pi}{k+1} \leq \tfrac{\pi}{n+1}$.

Moreover, if $L$ is monotone and orientable, then the Maslov numbers of each of the disks $f_1,\dots,f_k$ must be positive and even, and hence at least two.
Since these add up to $2c_1([\CP^1]) = 2(n+1)$, we conclude that $k=n+1$ and each of $f_1,\dots,f_k$ has Maslov number $2$.

Finally, suppose that $L$ is a torus but not necessarily monotone.
In this case we are not guaranteed that each $f_i$ has positive Maslov number.
On the other hand, if we assume that all relevant moduli spaces are regular, then the picture of our pseudoholomorphic building simplifies considerably, with no symplectization levels and each component in the top and bottom levels having index zero. 
In this case, some further index considerations show that each $f_i$ has Maslov number at most $2$, and hence at least $n$ of $f_1,\dots,f_k$ have Maslov number exactly equal to $2$.
To justify the regularity assumption, \cite{CM1,CM2} develops a detailed perturbation scheme based on domain dependent almost complex structures and curves with extra marked points constrained to lie on a Donaldson divisor (this forces the domains of all relevant curves to be stable).
\end{proof}

\begin{rmk}
Most of the technical difficulty in the above proof lies in the last part (i.e. the proof of Audin's conjecture), which is reflected in the long time span between that the papers \cite{CM1,CM2} and the original SFT paper \cite{EGH2000}.
\end{rmk}

\section{Extensions and further developments}\label{sec:further}

In this final section, we briefly outline various further directions in which the theory sketched above can be developed. Our list is by no means exhaustive, but it should at least convey the vast scope of symplectic field theory and its potential for future expansion. Some of these extensions are already discussed carefully in the original SFT papers and their immediate followups, while others are still under active development and/or are more speculative.

\subsection{Non-exact symplectic cobordisms, group ring coefficients, and twisted functoriality}\label{subsec:nonexact}

In the algebraic set up above we assumed that our cobordism $X$ is Liouville, so in particular its symplectic form is exact.
Among other things, this offers the useful simplification that the energy of an asymptotically cylindrical curve $C$ in $\wh{X}$ is determined via Stokes' theorem by the periods of its asymptotic Reeb orbits, independent of the homology class of $C$.
However, most of the analysis entering the SFT compactness theorem holds equally well if we relax the exactness condition on the symplectic form in the interior of $X$, letting $X$ be a symplectic cobordim whose boundary components are contact type hypersurfaces (this is sometimes called a ``strong symplectic cobordism'').
In this case, in order to compensate for the lack of a priori energy bounds and get well-defined SFT potentials we must incorporate additional group ring coefficients (suitably completed) which record the homology classes of our pseudoholomorphic curves.
Incidentally, working over the group ring can be useful even in the exact case, in order to probe more refined topological features on the target space.

\begin{rmk}\label{rmk:group_ring}
As explained in \cite[\S1.5]{EGH2000}, after making some additional topological assumptions and choices, we can associate to each asymptotically cylindrical curve in $\wh{X}$ an absolute homology class in $H_2(X)$. This is convenient for working over the group ring $\Q[H_2(X)]$ and passing to a completion with respect to the area functional $H_2(X) \ra \R$.
\end{rmk}

When $X$ is non-exact, it is no longer true that every asymptotically cylindrical pseudoholomorphic curve in $\wh{X}$ must have at least one positive end, so this somewhat complicates cobordism map functoriality. For example, the cobordism $X$ does not typically induce a map $C_\cha(Y_+) \ra C_\cha(Y_-)$ (think of the vertical flip of Figure~\ref{fig:bad_breaking}).
However, as pointed out by Cieliebak--Latschev following Fukaya \cite{fukaya2006application}, what we have instead is a twisted version of functoriality.
Namely, we can define a deformed version $\wt{\bdy}_\cha$ of the differential on $C_\cha(Y_-)$, and a DGA morphism $C_\cha(Y_+) \ra \wt{C}_\cha(Y_-)$, where the latter carries the deformed differential.
More specifically, the count of rigid planes in $\wh{X}$ with negative Reeb orbit asymptotics defines an element $\mm_{\wh{X}}$ in (a suitably completed group ring version of) $C_\cha(Y_-)$ which is a Maurer--Cartan element with respect to the RSFT $\Li$ structure (c.f. \S\ref{subsec:q-only}), and $\wt{C}_\cha(Y_-)$ is the corresponding deformation by $\mm_{\wh{X}}$ in the sense of Maurer--Cartan theory (see e.g. \cite{Yalin_MC} or \cite[\S2]{Fukaya-deformation}).

A similar twisted functoriality framework also exists for the $q$ variable only versions of 
rational and full SFT (as well as their linearizations), where the corresponding Maurer--Cartan elements take into account more general index zero curves in $\wh{X}$ with no positive ends. See e.g. \cite{fukaya2006application,cieliebak2009role,cieliebak2020homological,HSC} for more details and applications.

\subsection{Relaxing the contact condition}\label{subsec:relax_ctct_cond}

We can also relax the condition that $Y$ be a contact manifold, entering the wider world of \hl{stable Hamiltonian manifolds} (see \cite{cieliebak2015first}). 
Roughly, this means that $Y^{2n-1}$ carries a maximally nondegenerate two-form $\om$ and a one-form $\la$ such that $\la \wedge \om^{(n-1)} > 0$ and $\ker(\om) \subset \ker(d\la)$.
A typical example is given by the ``magnetic cosphere bundle'' $S^*Q$ of closed smooth manifold $Q$, where $\la$ is the canonical contact form and $\om = d\la + \pi^*\be$, with $\be$ a closed two-form on $Q$ and $\pi: S^*Q \ra Q$ projection to the base.
It is still possible to define the symplectization $\R \times Y$, but a key point is that its symplectic form need not be exact.
Similar to the contact case (i.e. when $\om = d\la$), there is a well-defined Reeb vector field $R$ characterized by $\om(R,-) = 0$ and $\la(R) = 1$.
Moreover, there is a class of almost complex structures on $\R \times Y$ which are \hl{symmetric, cylindrical, adjusted to $\om$}, for which the SFT compactness theorem naturally carries over to punctured curves with Reeb orbit asymptotics in $\R \times Y$ (see \cite[\S2]{BEHWZ}).
 Similarly, we can consider the completion $\wh{X}$ of a compact symplectic cobordism $X$ between stable Hamiltonian manifolds $Y_+$ and $Y_-$, and there is a corresponding well-behaved theory of asymptotically cylindrical curves in $\wh{X}$.

Remarkably, stable Hamiltonian structures allow us to view Floer theory as a special case of symplectic field theory. Namely, given a closed symplectic manifold $M$ with a time-dependent Hamiltonian $H: M \times S^1 \ra \R$, we consider the stable Hamiltonian structure on $M \times S^1$ with $\la = dt$ and $\om = dH \wedge dt$.
One can check that the Reeb orbits which wind once around the $S^1$ factor are in bijective correspondence with the $1$-periodic orbits of $H$, each time-dependent almost complex structure on $M$ corresponds to an SFT admissible almost complex on $\R \times M \times S^1$, and pseudoholomorphic cylinders in the latter precisely correspond to Floer cylinders in $M$.
This perspective immediately extends to give various generalizations of Floer theory, e.g. by replacing $M \times S^1$ with the mapping torus of some symplectomorphism $M \ra M$ (see e.g. \cite{fabert2009contact,fabert2020higher}), by allowing Reeb orbits with higher winding numbers, by incorporating more general rational or higher genus curves, etc.

For another interesting class of a stable Hamiltonian structures, let $Y$ be a principal circle bundle over a symplectic manifold $M$, with $\la$ a connection one-form and $\om$ the pull back of the symplectic form from $M$. The Reeb vector field $R$ is then the infinitesimal generator of the circle action.
Note that if the symplectic form on $M$ happens to be the curvature of the connection then $\la$ is in fact a contact form, and in this case the contact manifold $Y$ is called the \hl{pre-quantization} of the symplectic manifold $M$. 

Very recently, Fish--Hofer \cite{fish2023feral} have developed a theory of punctured pseudoholomorphic curves which may have infinite Hofer energy and need not be asymptotic to periodic orbits.
By a neck stretching argument applied to these so-called \hl{feral curves}, they are able to detect closed invariant sets in arbitrary compact regular level sets of a Hamiltonian $H: \R^4 \ra \R$.
It is interesting to ask how much of SFT can be extended to feral curves, with Reeb orbits replaced by more general closed invariant sets of Hamiltonian vector fields.


\subsection{Morse--Bott Reeb dynamics}\label{subsec:MB}

In the above, we typically assumed that our contact manifolds are equipped with contact forms $\al$ having nondegenerate Reeb dynamics, which in particular implies that there are only finitely many Reeb orbits of bounded period.
Although any contact form can be made nondegenerate by a small perturbation, many contact forms arising in nature enjoy extra symmetries which force the Reeb orbits to come in continuous families.
In the nicest case these Reeb orbit families are Morse--Bott, and rather than perturbing them away it is more natural to develop a Morse--Bott framework for symplectic field theory. Such a Morse--Bott approach to (cylindrical) contact homology was initiated in \cite{Bourgeois_MB} and subsequently studied in e.g. \cite{bourgeois2009exact,bourgeois2009symplectic,bourgeois2012effect}, with an especially promising approach to Morse--Bott analytic foundations appearing recently in \cite{yao2206cascades}.
A key point is that SFT compactness implies holomorphic curves with respect to approximating nondegenrate contact forms limit in the Morse--Bott setting to \hl{cascades}, which are hybrid objects combining punctured pseudoholomorphic curves with gradient flowlines of chosen Morse functions on the Morse--Bott loci.
It is natural to expect that all structures arising from SFT can be adapted to the Morse--Bott setting via a cascade approach.

\subsection{Evaluation constraints}\label{subsec:eval_constr}

We can also consider punctured pseudoholomorphic curves in a symplectization or symplectic cobordism which carry extra marked points, and then use evaluation maps to impose constraints at these marked points.
Among other things, this provides a natural way to incorporate curves of higher index into the SFT formalism, by cutting down their index with extra constraints.
The most standard approach is to require the marked points to pass through chosen cycles in the target space, in which case one expects the resulting invariants to depend only on the homology classes of these cycles.

For $i = 1,\dots,k$, let 
\[\ev_i: \ovl{\calM}_{g,k,A}^{\wh{X},J}(\Ga_+;\Ga_-) \ra \wh{X}\] denote the evaluation map at the $i$th marked point.
By analogy with Gromov--Witten theory, it tempting to pick cohomology classes $B_1,\dots,B_k \in H^*(\wh{X})$ and then integrate the cohomology class $\ev_1^*(B_1) \cup \cdots \cup \ev_k^*(B_k)$ over $\ovl{\calM}_{g,k,A}^{\wh{X},J}(\Ga_+;\Ga_-)$ to obtain a numerical invariant whenever the degree matches the dimension of the moduli space. 
However, an important complication is that, in contrast with Gromov--Witten theory, these compactified moduli spaces typically have boundary strata of expected codimension one, and hence the integral is not well-defined (even virtually) without more care.

Suppose first that we have compactly supported cohomology classes $B_1,\dots,B_k \in H^*_c(\wh{X})$, and choose $\mho_1,\dots,\mho_k \subset \wh{X}$ be Poincar\'e dual cycles.\footnote{Strictly speaking, we cannot always represent $\mho_1,\dots,\mho_k$ as smoothly embedded submanifolds, but we could use e.g. pseudocycles (see \cite[\S6.5]{JHOL}) or smooth simplicial chains. In \cite[\S2.3]{EGH2000}, the authors instead uses differential forms $\theta_1,\dots,\theta_k$ representing $B_1,\dots,B_k$, integrating $\ev_1^{-1}(\theta_1)\cup \cdots \cup \ev_k^{-1}(\theta_k)$ over the compactified moduli space $\ovl{\calM}_{g,k,A}^{\wh{X},J}(\Ga_+;\Ga_-)$. At a technical level, the latter requires making sense of integrals over (possibly very singular) compactified moduli spaces, while the former must deal with transversality of evaluation maps.}
Put 
\[\ovl{\calM}_{g,0,A}^{\wh{X},J}(\Ga_+;\Ga_-)\lll \mho_1,\dots,\mho_k\rrr := \ev_1^{-1}(\mho_1) \cap \cdots \cap \ev_k^{-1}(\mho_k),\]
which we interpret as the compactified moduli space of curves in $\ovl{\calM}_{g,0,A}^{\wh{X},J}(\Ga_+;\Ga_-)$ carrying the additional incidence constraints $\lll \mho_1,\dots,\mho_k\rrr$.

Recall that the DGA map $\Phi_\cha: C_\cha(Y_+) \ra C_\cha(Y_-)$ counts (at least heuristically) possibly disconnected curves in $\wh{X}$, such that each component has index zero and genus zero, with one positive end and many negative ends.
By instead counting the same curves but with the constraint $\lll \mho_1,\dots,\mho_k\rrr$ distributed amongst the components, we get chain map 
$\Phi_\cha\lll \mho_1,\dots,\mho_k\rrr: C_\cha(Y_+) \ra C_\cha(Y_-)$.
Note that this map is not a DGA morphism (i.e. it is not multiplicative) due to the way the constraints are distributed, but we can modify it to become one by allowing each constraint to repeat arbitrarily many times.
Namely, introduce a formal variable $t_i$ for each $\mho_i$, and define a $\Q\ll t_1,\dots, t_k\rr$-linear DGA morphism $\Phi_\cha \lll \mho_1^\bullet,\dots,\mho_k^\bullet\rrr: C_\cha(Y_+)\ll t_1,\dots,t_k\rr \ra C_\cha(Y_-)\ll t_1,\dots,t_k\rr$ by
\begin{align*}
\Phi_\cha \lll \mho_1^\bullet,\dots,\mho_k^\bullet\rrr(x) := \sum\limits_{j_1,\dots,j_k \geq 0} \Phi_\cha \lll \underbrace{\mho_1,\dots,\mho_1}_{j_1},\dots,\underbrace{\mho_k,\dots,\mho_k}_{j_k}\rrr(x) \,t_1^{j_1}\cdots t_k^{j_k},
\end{align*}
for $x \in C_\cha(Y_+)$.
By a homotopy argument, the induced map on homology should depend only on the compactly support cohomology classes $B_1,\dots,B_k$.\footnote{At a basic level, the idea is that given homologous cycles $\mho,\mho'$ and a chain $\wt{\mho}$ with $\bdy \wt{\mho} = \mho - \mho'$, the count of curves with evaluation constraints in $\wt{\mho}$ should give a chain homotopy between $\Phi_\cha \lll \mho \rrr$ and $\Phi_\cha \lll \mho' \rrr$.}

Similarly, for symplectization curves we have evaluation maps 
\[\ev_i: \ovl{\calM}_{g,k,A}^{\R \times Y,J}(\Ga_+;\Ga_-)/\R \ra Y,\]
and for cycles $\mho_1,\dots,\mho_k$ in $Y$ we define a deformed differential
$\bdy_\cha \lll \mho_1^\bullet,\dots,\mho_k^\bullet \rrr: \calA_{Y_+}\ll t_1,\dots,t_k\rr \ra \calA_{Y_-}\ll t_1,\dots,t_k\rr$
which roughly counts index one rational curves in $\R \times Y$ with one positive end and many negative ends, with extra marked points mapping to various copies of $\mho_1,\dots,\mho_k$, modulo target translations.
The homology of the resulting DGA depends only on the homology classes of $\mho_1,\dots,\mho_k$.

We can also combine the above two pictures by considering noncompact cycles $\mho_1,\dots,\mho_k$ in $\wh{X}$ which are cylindrical at infinity and represent noncompactly supported cohomology classes $B_1,\dots,B_k \in H^*(\wh{X})$.
Restricting these to the ends and then projecting gives cycles $\mho_1|_{Y_\pm},\dots,\mho_k|_{Y_\pm}$ in $Y_\pm$.
Then as above we can deform the DGAs $C_\cha(Y_\pm)$ by counting curves in $\R \times Y_\pm$ with evaluation constraints in $\mho_1|_{Y_\pm},\dots,\mho_k|_{Y_\pm}$ in $Y_\pm$, and counting curves in $\wh{X}$ with evaluation constraints in $\mho_1,\dots,\mho_k$ gives a DGA morphism $\Phi_\cha \lll \mho_1^\bullet,\dots,\mho_k^\bullet \rrr$ between these deformed DGAs.

The above discussion also naturally extends to rational and full SFT. For instance, in the $q$ variable only approach we get a deformed differential $\bdy_\sft^\qonly \lll \mho_1|_{Y_\pm},\dots,\mho_k|_{Y_\pm} \rrr$ on $\calA_{Y_\pm}\ll \hbar ,t_1,\dots,t_k\rr$ which gives rise to a deformed chain complex $C_\sft^\qonly(Y_\pm) \lll \mho_1^\bullet,\dots,\mho_k^\bullet \rrr$,
along with a deformed cobordism map $\Phi_\sft^\qonly: C_\sft^\qonly(Y_+) \lll \mho_1^\bullet,\dots,\mho_k^\bullet\rrr \ra C_\sft^\qonly(Y_-) \lll \mho_1^\bullet,\dots,\mho_k^\bullet\rrr$.
The claim is that on the homology level (or more precisely up to suitable chain homotopies) these depend only on the cohomology classes of $\mho_1,\dots,\mho_k$.

\begin{rmk}\label{rmk:U_map}
As a special case of the above, we can consider the trivial symplectic cobordism $X = Y \times [0,1]$. Note that $\wh{X}$ is identified with $\R \times Y$, but here we are ignoring the $\R$ action by target translations.
Then for any cycle $\mho$ in $Y$, we get a DGA endomorphism $\Phi_\cha\lll \mho^\bullet \rrr: C_\cha(Y) \ll t \rr \ra C_\cha(Y)\ll t \rr$.
When $\mho$ is a point, this is analogous to the $U$ map in embedded contact homology (see \cite[\S3.8]{Hlect}).
\end{rmk}
\begin{rmk}
Suppose that $Z$ is a smoothly embedded codimension two contact submanifold of $Y$ (e.g. a transverse knot in the three-sphere).
If we work with an almost complex structure $J$ on $\R \times Y$ which preserves $\R \times Z$, then all $J$-holomorphic curves in $\R \times Y$ must intersect $\R \times Z$ nonnegatively (except for possibly those contained in $\R \times Z$). By recording these homological intersection numbers, \cite{cote2024homological} defines a deformation of $C_\cha(Y)$ which is sensitive to the contact isotopy class of $Z$.

\end{rmk}

\subsection{Local geometric constraints}\label{subsec:geom_constr}

There are many other ways to impose geometrically meaningful constraints on punctured pseudoholomorphic curves.
Compared with the evaluation constraints discussed in \S\ref{subsec:eval_constr}, these may have advantages in terms of what types of buildings they can degenerate into, how computable they are, and so on.
For simplicity, let us focus on constraints for punctured curves in $\wh{X}$ which are localized near a point $p$ in the target space.
The simplest example is a point constraint $\lll p \rrr$, which played a central role in the above proof of Theorem~\ref{thm:hyp}.
Meanwhile, the proof of Theorem~\ref{thm:CM} utilized local tangency constraints $\lll \T_D^{(m)}p\rrr$, which were defined in \cite{CM2} and further studied in e.g. \cite{tonk,McDuffSiegel_counting}.
Compared with several ordinary point constraints $\lll p_1,\dots,p_k\rrr$, local tangency constaints have the advantage that they must be carried by a single component in any limiting pseudoholomorphic building, whereas the constraint $\lll p_1,\dots,p_k\rrr$ could become divided amongst the components.\footnote{Strictly speaking local tangency constraints also get divided amongst nearby components when the component carrying the marked point becomes a ghost (i.e. constant), but we will gloss over this technicality here.}

One can also consider \hl{blowup constraints} by considering curves lying in various homology classes in the blowup of $\wh{X}$ at $p$. Under the projection map to $\wh{X}$, these can be interpreted (at least heuristically) as curves in $\wh{X}$ which have several branches passing through $p$.
We could also further impose a local tangency constraint on each branch passing through $p$, giving \hl{multibranched tangency constraints} (see \cite[\S2.3]{McDuffSiegel_counting}).
Or, we can instead consider \hl{multidirectional tangency constraints} $\lll \T_{D_1}^{(m_1)}\cdots \T_{D_n}^{(m_n)}p\rrr$ by imposing tangency orders $m_1,\dots,m_n$ on a single branch of the curve but with respect to several generic local divisors $D_1,\dots,D_n$ at $p$, which roughly forces our curves to have a cusp singularity modeled on $t \mapsto (t^{m_1},\dots,t^{m_n})$ at the point $p$ (see \cite[\S3]{cusps_and_ellipsoids}).

Alternatively we can take a small neighborhood $U$ of $p$ with smooth boundary and consider curves in the symplectic completion of $X' := X \setminus U$, which we view as having positive boundary $Y_+$ and negative boundary $Y_- \sqcup \bdy U$.
Specifying the negative Reeb orbit asymptotics in $\bdy U$ is akin to imposing additional (a priori less geometric) constraints on curves in $\wh{X}$.
In fact, by neck stretching along $\bdy U$, we can convert any geometric constraint localized near $p$ to a (possibly very complicated) linear combination of extra negative end constraints in $\bdy U$ as above.
The precise correspondence depends on the symplectomorphism type of $U$, which we could take to be a round ball or a more general ellipsoid boundary $\bdy E(a_1,\dots,a_n)$ for chosen $a_1,\dots,a_n \in \R_{>0}$.
For instance, in the case $n=2$, when $U$ is a skinny ellipsoid $E_\op{sk} = E(a_1,a_2)$ with $a_2 \gg a_1$, extra negative ends in $\bdy U$ closely match up with local tangency constraints at $p$ (see \cite[\S4.1]{McDuffSiegel_counting}). More generally, for $U = E(\eps a_1,\eps a_2)$ with coprime $a_1,a_2 \in \Z_{\geq 1}$ and $\eps > 0$ small, extra negative ends in $\bdy U$ roughly agree with the multidirectional tangency constraint $\lll \T_{D_1}^{(a_1)}\T_{D_2}^{(a_2)} p\rrr$ (see \cite[\S3]{cusps_and_ellipsoids}).

\subsection{Gravitational descendants}

While the SFT invariants discussed in \S\ref{sec:formalism} are based on counting curves with all possible conformal structures on the domain, we can also try to define more refined invariants by imposing restrictions on these conformal structures.
Using the forgetful map $\ovl{\calM}^{\wh{X},J}_{g,k,A}(\Ga_+;\Ga_-) \ra \ovl{\calM}_{g,k+s_+ + s_-}$, where $\Ga_\pm$ consists of $s_\pm$ Reeb orbits, we could restrict our counts to curves lying over some chosen cycle in the Deligne--Mumford space $\ovl{\calM}_{g,k+s_+ + s_-}$.
By analogy with Gromov--Witten theory, a natural approach to imposing such constraints would be to adapt the construction of $\psi$ classes (see e.g. \cite{kock2001notes}) to the setting of SFT moduli spaces.
Namely, for $i=1,\dots,k$, let $\calL_i$ be the complex line bundle over $\calM^{\wh{X},J}_{g,k,A}(\Ga_+;\Ga_-)$ whose fiber over a curve $C$ is the cotangent line of the domain Riemann surface of $C$ at its $i$th marked point.
Assuming this line bundle extends over the compactification, we define $\psi_i \in H^2(\ovl{\calM}^{\wh{X},J}_{g,k,A}(\Ga_+;\Ga_-))$ to be its first Chern class.
We can then try to integrate the cup product $\psi_1^{j_1} \cup \cdots \cup \psi_k^{j_k}$ over $\ovl{\calM}^{\wh{X},J}_{g,k,A}(\Ga_+;\Ga_-)$ for some $j_1,\dots,j_k \in \Z_{\geq 0}$, possibly after imposing evaluation constraints $\mho_1,\dots,\mho_k$ at the marked points.

However, similar to the discussion in \S\ref{subsec:eval_constr}, this runs into serious complications stemming from the fact that $\ovl{\calM}^{\wh{X},J}_{g,k,A}(\Ga_+;\Ga_-)$ has codimension one boundary strata.
In some special cases it may possible to rule out codimension one boundary strata (e.g. this holds for symplectic ellipsoids by Conley--Zehnder index considerations), but in general we must choose specific differential forms representing $\psi_1,\dots,\psi_k$ with carefully prescribed behavior over the boundary strata. One proposal for doing so appears in \cite{Fabert_descendants_2011,fabert2011string}.
Packaged together, these should give various gravitational descendant cobordism maps, for instance
$\Phi_\cha\lll \psi^{j_1}\mho_1,\dots,\psi^{j_k}\mho_k\rrr: C_\cha(Y_+) \ra C_\cha(Y_-)$ and its RSFT and SFT extensions.

\subsection{Relative symplectic field theory}\label{subsec:relSFT}

The open string analogue of SFT, called \hl{relative symplectic field theory}, assigns algebraic invariants to Legendrian submanifolds of contact manifolds and Lagrangian cobordisms between them.
More precisely, we consider pairs $(Y_\pm,\La_\pm)$, with $Y_\pm^{2n-1}$ contact manifolds and $\La^{n-1} \subset Y_\pm$ Legendrian submanifolds, and also pairs $(X,L)$, where $X$ is a symplectic cobordism between $Y_+$ and $Y_-$ and $L \subset X$ is a Lagrangian submanifold with $L \cap Y_\pm = \La_\pm$.
The invariants are defined in terms of proper pseudoholomorphic maps $(\Si,\bdy \Si) \ra (\R \times Y,\R \times \La_\pm)$ and $(\Si,\bdy \Si) \ra (\wh{X},\wh{L})$, where $\Si$ is a Riemann surface with boundary which has both interior and boundary punctures, and $\wh{L}$ is given by attaching cylindrical Lagrangian ends to $L$. 
As in the absolute case, the interior punctures are asymptotic to closed Reeb orbits in $Y_\pm$, while the boundary punctures are asymptotic to Reeb chords of the Legendrians $\La_\pm$.

The Legendrian analogue of $C_\cha(Y)$ is roughly the (noncommutative) DGA generated by the Reeb chords of $\La$, with differential counting index one disks in $(\R \times Y,\R \times \La)$ with one positive boundary puncture and many negative boundary punctures, modulo target translations.
In general this must be taken as a module over $C_\cha(Y)$ due to the possibility of bubbling off planes, but in the presence of a filling $X$ of $Y$ (or an abstract augmentation) we can instead count disks with boundary punctures in $\R \times Y$ with extra anchors in $X$.
For Legendrian links of $\R^3$, Chekanov \cite{Chekanov_DGA} gave a purely combinatorial construction of the Legendrian contact homology algebra as a DGA, which he famously used to distinguish two Legendrian $5_2$ knots (up to Legendrian isotopy) which could not be distinguished by classical methods.
In higher dimensions, a version of the Legendrian contact homology algebra for Legendrians in $\R^{2n+1}$ appears in \cite{ekholm2005contact}.

When considering more general rational or higher genus curves with boundary, a new complication called \hl{string degeneration} arises, which is that a chord $(I,\bdy I) \subset (\Sigma,\bdy \Sigma)$ could get contracted down to a point, thereby degenerating $\Si$ into a (possibly disconnected) Riemann surface with two boundary points pinched together (see e.g. \cite[Fig. 18]{cieliebak2009role}).
Thus we must either make some additional assumptions which rule out string degenerations, or else add terms to the differential which account for these degenerations. 
An approach to defining relative RSFT by incorporating operations from string topology is sketched in \cite[App. A]{cieliebak2009role}, while a different approach based on taking multiple copies of $\La$ to rule out string degenerations between different components is given in \cite{ekholm2008rational}.
In case of links in $\R^3$, Ng \cite{ng2010rational} has given a fully combinatorial model for relative RSFT in the spirit of Chekanov's approach.
To our knowledge, a detailed algebraic formalism for relative SFT in full genus has not yet appeared in the literature.

\subsection{Non-equivariant SFT and comparisons with Floer theory}\label{subsec:ext:Floer}

As we pointed out in \S\ref{subsec:relax_ctct_cond}, Floer homology can be viewed as a special case of symplectic field theory for stable Hamiltonian structures.
At the same time, many SFT invariants for Liouville domains $X$ and their Lagrangian submanifolds $\La$ 
are expected to have isomorphic counterparts in Floer theory.
Indeed, \cite{bourgeois2009exact,bourgeois2015erratum} discusses an isomorphism between the linearized contact homology $H(C_\chlin(X))$ and positive $S^1$-equivariant symplectic cohomology $\sh_{S^1,+}(X;\Q)$ with rational coefficients.
Here ``positive'' means roughly that we quotient out by those constant Hamiltonian orbits which contribute to the ordinary homology of $X$, although \cite{Bourgeois-Oancea_equivariant} also discusses an enlarged ``filled'' version of $C_\chlin(X)$ whose homology should correspond to the full $S^1$-equivariant symplectic cohomology $\sh_{S^1}(X;\Q)$.
Also, the fact that we get $S^1$-equivariant symplectic cohomology reflects the fact that symplectic field theory is ``by default'' $S^1$-equivariant, i.e. it is generated by unparametrized Reeb orbits, whereas Hamiltonian Floer homology is generated by {\em parametrized} Hamiltonian orbits.
However, following \cite{bourgeois2009exact,HN_hypertight}, it is also possible to define a nonequivariant version of $C_\chlin(X)$, in which each Reeb orbit $\ga$ is viewed as an $S^1$-family of parametrized Reeb orbits which then contributes two generators $\wh{\ga},\wc{\ga}$ in a Morse--Bott model (these correspond to the maximum and minimum of a perfect Morse function on $\ga$).
In other words, each flavor of linearized contact homology is expected to match up with a corresponding version of symplectic cohomology.

Next, we can ask whether the higher parts of SFT (i.e. $C_\cha,C_\rsft,C_\sft$) have counterparts on the Floer side. 
Corresponding to the contact homology algebra $C_\chalin(X)$, Ekholm--Oancea \cite{Ekholm-Oancea_DGAS} constructed an analogous CDGA structure extending positive equivariant symplectic cochains $\sc_{S^1,+}(X)$ (as well as nonequivariant and relative versions).
We could alternatively view this CDGA as the cobar construction applied to an $\Li$ coalgebra structure with underlying chain complex $\sc_{S^1,+}(X)$.
A key insight in \cite{Ekholm-Oancea_DGAS} is to study the Floer equation on Riemann spheres with one positive and many negative punctures, where the Floer-theoretic weights at the negative ends are allowed to vary over a simplex.
The resulting operations are ``secondary'' in the sense that the operations with fixed weights are essentially trivial, whereas the ones with varying weights have shifted degrees which match up with the corresponding terms for $C_\chalin(X)$.
It seems plausible that the rational symplectic field theory of $X$ also has an extended counterpart in symplectic cohomology, defined in terms of genus zero Floer-theoretic curves with many positive and many negative punctures, with simplex-varying weights at the negative ends.
Similarly, one can ask for a higher genus version of symplectic cohomology which serves as a counterpart for full SFT.

This bridge between SFT and Floer theory can be useful in several ways. For one thing, as above it suggests various refinements of symplectic cohomology based on known or expected algebraic structures in symplectic field theory.
At the same time, as mentioned in \S\ref{sec:transversality}, it suggests that we could use higher algebraic structures in symplectic cohomology as an ersatz for SFT and as a way to circumvent transversality difficulties. It is arguably a virtue that certain constructions are more transparent in SFT (e.g. neck stretching, $S^1$-equivariance)  while others are more natural in Floer theory (Hamiltonian spectral invariants, filled versions).
Incidentally, all of the algebraic formalism discussed in \S\ref{sec:formalism} should have corresponding ``hat check'' nonequivariant versions, and these ought to recover the default $S^1$-equivariant SFT invariants by a suitable notion of algebraic $S^1$ quotient.

\subsection{Embedded contact homology}\label{subsec:ECH}

Embedded contact homology (ECH) is an analogue of SFT for three-dimensional contact manifolds 
which is also defined using asymptotically cylindrical punctured pseudoholomorphic curves in symplectizations, but which restricts to only those curves which are embedded (roughly speaking).
By construction, it is isomorphic to a version of three-dimensional Seiberg--Witten Floer homology, giving a three-dimensional analogue of Taubes' theorem \cite{taubes2000seiberg} equating the Seiberg--Witten invariants of a four-dimensional symplectic manifold $M$ with a count of (roughly) embedded pseudoholomorphic curves in $M$.
Although ECH is defined using many of the same basic ingredients as full SFT for contact three-manifolds $Y$, it also has many important differences, for instance it depends only on the diffeomorphism type of $Y$ (although there is a \hl{contact element} which is sensitive to the contact structure -- see \cite[\S2.2]{Hutchings_quantitative_ECH}).
The ECH differential is defined rigorously in the foundational papers directly using holomorphic curves, while the cobordism map is only defined (at least the time of writing) using the isomorphism with Seiberg--Witten Floer homology.\footnote{To give a somewhat tongue-in-check quote by Michael Hutchings, ``whereas embedded contact homology is only defined in dimension three, symplectic field theory is not defined in any dimension''.}
The ECH chain complex does not have a product structure, but it does have a special grading coming from the so-called ECH index, and it also enjoys a $U$ map (c.f. Remark~\ref{rmk:U_map}).
Embedded contact homology has given many impressive applications to low dimensional contact and symplectic geometry, 
 although its precise relationship with SFT is still not fully understood (see \cite{hutchings_ECH_and_SFT} for more details).

\subsection{Quantitative invariants}\label{subsec:quant}

By default, the algebraic invariants discussed in \S\ref{sec:formalism} are qualititative invariants of contact manifolds and symplectic cobordisms, which means in particular that as homology elements they depend only on $Y$ up to contactomorphism and on $X$ up to Liouville homotopy (see \cite[\S11.2]{cieliebak2012stein} and \cite[\S2.6]{EGH2000}).
However, since our punctured curves have nonnegative energy which is determined (at least in the exact case) by the actions of the positive and negative asymptotic Reeb orbits, it follows that our invariants come with a natural $\R$-filtration by action, and this is preserved by the differentials and other algebraic operations.
By incorporating this filtration, we can use SFT to define refined numerical invariants which are sensitive to quantitative features of contact forms and symplectic forms.
	
One popular application is to obstruct symplectic embeddings between star-shaped domains in $\R^{2n}$. In dimension $2n=4$, the idea of using filtered embedded contact homology to obstruct symplectic embeddings between domains in $\R^4$ was pioneered by Hutchings in \cite{hutchings2016beyond}. In particular, it was shown in \cite{McDuff_Hofer_conjecture} that the symplectic capacities defined using ECH give a complete set of obstructions for symplectic embeddings between four-dimensional ellipsoids.

In higher dimensions, new obstructions based on SFT moduli spaces were discovered in \cite{HK}, \cite{hind2015some}. These were formalized and extended to symplectic capacities based on filtered SFT in \cite{HSC}, following analogous capacities from Floer homology constructed in \cite{Gutt-Hu}.
More recent applications of filtered SFT to Hamiltonian and Reeb dynamics have also appeared e.g. in \cite{hutchings2022elementaryspectral,edtmair2022elementary,chaidez2022contact,chaidez2023elementary}.

Quantitative embedding problems provide a compelling testing ground for Eliashberg's ``holomorphic curves or nothing'' metaprinciple, which posits that any symplectic geometric construction not obstructed by pseudoholomorphic curves or classical topology should in fact exist. 
For example, this would suggest that there exists a symplectic embedding from one ellipsoid $E := E(a_1,\dots,a_n)$ to another one $E' := E(a_1',\dots,a_n')$ unless it is obstructed by some pseudoholomorphic curve or by volume considerations (which require $a_1\cdots a_n \leq a_1' \cdots a_n'$).
To understand what it means to be obstructed by a pseudoholomorphic curve in this setting, observe first that there is an inclusion $E(\eps a_1,\dots,\eps a_n) \subset E(a_1',\cdots,a_n')$
for $\eps > 0$ sufficiently small. Let $X$ be the complementary cobordism, with contact boundaries $Y_+,Y_-$.
For degree parity reasons, the differentials on $C_\cha(Y_\pm),C_\rsft^\qonly(Y_\pm),C_\sft^\qonly(Y_\pm)$ all vanish identically, and the action filtrations are easy to write down combinatorially (recall Ex~\ref{ex:mult_cov}).\footnote{We will assume for simplicity that the area factors $(a_1,\dots,a_n)$ of $E$ are rationally independent, so that $\bdy E$ has nondegenerate Reeb dynamics, and similarly for $\bdy E'$.}
This essentially means that the cobordism maps $\Phi$ induced by $\wh{X}$ are deformation invariant and they see all index zero curves in $\wh{X}$ (or rather those with nonzero algebraic counts). 
In particular, given a hypothetical symplectic embedding of $E$ into $E'$, its complementary cobordism would be Liouville homotopic to $X$, and thus its cobordism map would agree with $\Phi$.
The upshot is that for each nonzero term of $\Phi$ we can read an obstruction to symplectically embedding $E$ into $E'$, and together these are expected to give a complete list of obstructions (along with volume, and possibly incorporating other geometric constraints and so on).
This leads to the following open problem, which highlights how subtle computations in symplectic field theory can be, even in seemingly simple examples.

\begin{problem}\label{prob:ell_cob_map}
Given an inclusion of ellipsoids $E(a_1,\dots,a_n) \subset E(a_1',\cdots,a_n')$, let $X$ denote the complementary cobordism, with contact boundaries $Y_+,Y_-$.
Compute the corresponding RSFT potential $\ff_{\wh{X}}$ and full SFT potential $\FF_{\wh{X}}$.
\end{problem}
\NI Ideally, a solution to Problem~\ref{prob:ell_cob_map} should entail an effective procedure for determining whether the coefficient of a given monomial in $\ff_{\wh{X}}$ or $\FF_{\wh{X}}$ is nonzero.

\subsection{Invariants of contact domains}\label{subsec:ctct_domains}

Eliashberg--Kim--Polterovich \cite{eliashberg2006geometry} defined a version of cylindrical contact homology for domains $U \subset\R^{2n} \times S^1$, and they used this to prove a contact geometric analogue of Gromov's nonsqueezing theorem.
Their invariant is defined roughly as a direct limit of the cylindrical contact homologies of $\R^{2n} \times S^1$ in the action window $(0,\eps)$ and restricted to Reeb orbits which wind once around the $S^1$ factor, where the colimit is with respect to a sequence of contact forms which approach zero inside of $U$ and stay fixed outside of $U$.
This invariant has the useful feature that for domains of the form $U = V \times S^1$, with $V \subset \R^{2n}$, it is isomorphic to a version of the symplectic homology of $V$ in the action window $(-\infty,-1)$.
A $\Z/p$-equivariant analogue was shown in \cite{fraser2016contact} to give stronger contact nonsqueezing results, and various related invariants have also been defined using microlocal sheaves \cite{chiu_nonsqueezing,zhang2021capacities} and generating functions \cite{sandon2011contact,FSB2023contact}.
It is natural to ask whether deeper layers of SFT (e.g. $C_\cha,C_\rsft,C_\sft$) could be used to define invariants of domains $U$ in $\R^{2n} \times S^1$ (or in more general contact manifolds), and whether these could detect more refined versions of contact nonsqueezing.

\subsection{Integrable systems}

It is observed in \cite[\S2.2]{eliashberg_sft_and_applications} that the symplectic field theory of the simplest contact manifold, namely the circle, naturally produces an infinite system of commuting integrals of the dispersionless Korteweg--de Vries equation (KdV) $u_t + u u_x = 0$, with higher genus curves relating to its quantization.
More generally, it is expected that rational symplectic field theory associates to any circle bundle over a closed symplectic manifold an infinite dimensional integrable Hamiltonian partial differential equation, with full symplectic field theory giving its quantization (see e.g. \cite{dubrovin2016symplectic, fabert2011string,kluck2010symplectic}).
While this connection has been worked out in detail in specific examples, its full ramifications for symplectic field theory and Hamiltonian PDEs remains to be explored.

\subsection{Effect of Weinstein handle attachment}

Recall that most known examples of Liouville domains are Weinstein, which means that $X^{2n}$ is built up from the ball by attaching various subcritical (index less than $n$) Weinstein handles and critical (index equal to $n$) Weinstein handles.
There is a general expectation that most (qualitative) pseudoholomorphic curve invariants $X$ are unchanged by subcritical handle attachment, while the key symplectic topological features of $X$ are encoded by the attaching Legendrian spheres $\La_1,\dots,\La_k$ of the critical handles.
In particular, subcritical Weinstein domains are governed by an h principle, their symplectic cohomology vanishes, and by \cite{yau2004cylindrical} the cylindrical contact homology of the contact boundary is determined by the ordinary homology of $X$.

In \cite{bourgeois2012effect} and the followup paper \cite{bourgeois2011symplectic}, the authors give various formulas for pseudoholomorphic curve invariants of $X$ in terms of pseudoholomorphic invariants of $\La_1,\dots,\La_k$.
In particular, they describe the linearized contact homology $C_\chlin(X)$ in terms of the cyclic homology of the Legendrian contact homology algebra of the link $\La_1 \cup \cdots \cup \La_k$.\footnote{More precisely, this describes an extended version of $C_\chlin(X)$ with additional generators coming from the ordinary homology of $X$.}
It is quite desirable to extend these formulas in order to describe the higher SFT invariants $C_\cha(Y),C_\rsft(Y),C_\sft(Y)$ in terms of the relative SFT of $\La_1 \cup \cdots \cup \La_k$, as this could open up the possibility of computing symplectic field theory for many interesting contact manifolds.
In particular, in dimension $2n=4$ one might hope for a purely combinatorial formula for these invariants in terms of polygons in a Legendrian link diagram (c.f. \cite{ekholm2015legendrian}).

\subsection{Relationship with relative Gromov--Witten theory}\label{subsec:rel_GW}

An important class of examples coming from algebraic geometry arises when $M$ is a smooth complex projective variety and $D$ is a nonsingular ample divisor.
We can find a small tubular neighborhood $U$ of $D$ whose boundary $Y := \bdy U$ is a contact type hypersurface in $M$, where $Y$ is contactomorphic to a pre-quantization of the symplectic manifold $D$ (recall \S\ref{subsec:relax_ctct_cond}).
In particular, $Y$ is foliated by closed Reeb orbits and we can consider its Morse--Bott symplectic field theory as in \S\ref{subsec:MB}. Alternatively, after a small perturbation of the contact form, there is one simple nondegenerate Reeb orbit for each critical point of a chosen Morse function on $D$.
Moreover, the complement $X := \ovl{M \setminus D}$ carries the structure of a Liouville domain (in fact a Stein domain -- see e.g. \cite[\S 4b]{Seidel_biased_view}), while the symplectic completion $\wh{U}$ is identified with the total space of the normal bundle of $D$ in $M$.
In particular, generalizing the example of a line in the complex projective plane from \S\ref{sec:glance}, we can work with an almost complex structure on $\wh{U}$ for which the projection to $D$ is holomorphic, and thus we can understand punctured pseudoholomorphic curves in both $\wh{U}$ and $\R \times Y$ in terms of closed pseudoholomorphic curves in $D$ and their meromorphic lifts.

One can show by index considerations that there are no contributing index one punctured curves in $\R \times Y$ (at least in the absence of any additional constraints), so that the invariants $C_\cha(Y),C_\rsft(Y),C_\sft(Y)$ all have trivial differentials.
Note that we can view an asymptotically cylindrical pseudoholomorphic curve in $\wh{X}$ as a punctured curve in $M$ with removable singularities, such that after filling in the punctures a positive end asymptotic to a Reeb orbit of multiplicity $\ka$ corresponds to a point intersecting $D$ with contact order $\ka$.
By translating punctured curves into closed curves in this way, the SFT compactified moduli space of punctured curves in $\wh{X}$ is closely related to the moduli space of stable relative maps in $(M,D)$ used to define relative Gromov--Witten theory.
In fact, as observed in \cite[Rmk. 5.9]{BEHWZ}, one can essentially view the compactness theorems proved in \cite{Ionel-Parker_relative_GW,IP2004symplecticsum,Li-Ruan_symplectic_surgery} as special cases of the SFT compactness theorem. 
In particular, the type of decomposition along a divisor appearing in the symplectic sum formula can really be viewed as a special case of stretching the neck.
However, an important subtlety is that the relative stable maps compactification considers neck levels modulo an action by $\C^*$ rather than $\R$, and thereby has only boundary strata of (real) expected codimension $2$. This allows one to define relative Gromov--Witten invariants taking values in rational numbers (as opposed to say chain complexes), whereas the SFT compactification a priori has codimension $1$ boundary strata.
This is a recurring theme in symplectic field theory: individual structure coefficients are not a priori invariant under changes of the almost complex structure and other data, although in certain nice situations they may turn out to have stronger invariance properties.

\subsection{Normal crossings divisors and extended field theory structure}

It is natural to try to extend the discussion in \S\ref{subsec:rel_GW} by allowing the divisor $D$ to have normal crossings singularities (e.g. a nodal algebraic curve in a smooth complex projective surface).
In this situation we can still find a small neighborhood $D$ whose boundary is a smooth contact hypersurface $Y$, with the Reeb dynamics on $Y$ controlled but much more complicated than in the pre-quantization case.
Roughly, with respect a natural stratification on $D$, the Reeb orbits in $Y$ come in various families which are $\TT^{r-1}$ torus bundles over open strata $\calS$, where $\calS$ has dimension $n-r$ and where $2n = \dim_\R M$ (see \cite{mclean2012growth,tehrani2018normal,mcleanslides}).
Putting $X := \ovl{M \setminus U}$ as before, we find that the SFT of $\wh{X}$ is closely related to the relative Gromov--Witten invariants of $(M,D$) as in \cite{Ionel_rel_GW_NC}.

Pairs $(M,D)$ as above with $D$ a normal crossings divisor arise naturally when considering multiple cuts as in \cite[\S1.1]{venugopalan2020tropical}, which can be thought of as a multidirectional generalization of neck stretching along several intersecting hypersurfaces.
Note that the multiple cut reduces to the usual symplectic cut as defined by Lerman \cite{lerman1995symplectic} in the case of a single smooth hypersurface which is the fiber of the moment map for a Hamiltonian circle action.
Gluing and compactness results for pseudoholomorphic curves along multiple cuts are discussed in the recent manuscript \cite{venugopalan2020tropical} and in the work of Brett Parker on exploded manifolds (see e.g. \cite{parker2015gluing,parker2016notes}). 
Here instead of pseudoholomorphic buildings one encounters more complicated configurations (indexed by tropical graphs) of curves with matched asymptotics in various target spaces arising from the cut.

A closely related question asks whether we extend the possibilities for gluings in symplectic field theory by assigning invariants to contact manifolds with convex boundary and to suitable symplectic cobordisms between these. One proposal for defining the contact homology of contact manifolds with boundary  appears in \cite{colin2011sutures} using the language of sutures. 
In a similar vein, we can also ask to what extent SFT invariants glue together under decompositions of Liouville manifolds into Liouville sectors as in \cite{ganatra2020covariantly,ganatra2024sectorial}.

\subsection{Connections with string topology and smooth manifold invariants}

Another important class of examples comes from smooth topology.
Given a closed smooth manifold $Q$, its unit disk cotangent bundle $D^*Q$ (with respect to any Riemannian metric) is a Liouville domain whose contact boundary is the unit sphere cotangent bundle $S^*Q$, and the symplectic completion of $D^*Q$ is identified with the full cotangent bundle $T^*Q$.
Given two closed smooth manifolds $Q_1$ and $Q_2$ which are diffeomorphic, it is easy to check that $T^*Q_1$ and $T^*Q_2$ are symplectomorphic and $S^*Q_1$ and $S^*Q_2$ are contactomorphic.
The converse is a central question in symplectic topology known as the \hl{weak nearby Lagrangian conjecture}, which states that $Q_1$ and $Q_2$ should be diffeomorphic if their cotangent bundles $T^*Q_1$ and $T^*Q_2$ are symplectomorphic.
Meanwhile, Eliashberg's metaprinciple (c.f. \S\ref{subsec:quant}) posits that if $T^*Q_1$ and $T^*Q_2$ are not symplectomorphic then there must be some pseudoholomorphic curve invariant (or classical invariant) which distinguishes them.
Since symplectic field theory in some sense knows about all punctured curves in $T^*Q$ and $\R \times S^*Q$, this suggests that some suitably enhanced version of SFT should distinguish between $T^*Q_1$ and $T^*Q_2$.

In Floer theory, recall that Viterbo's isomorphism \cite{viterbo1999functors,salamon2006floer} identifies the symplectic cohomology of $T^*Q$ with the homology of the free loop space $\calL Q$ of $Q$ (possibly up to a twist of coefficients).
This isomorphism is known to respect many algebraic operations, e.g. the pair of pants product on symplectic cohomology matches up with the Chas--Sullivan product from string topology (see \cite{abbondandolo2006floer}), and this extends to the full BV algebra structures on both sides (see \cite{abouzaid2013symplectic}).
In light of the discussion in \S\ref{subsec:ext:Floer}, one should expect similarly close connections between SFT and string topology.
Indeed, an isomorphism between linearized contact homology $H(C_\chlin(T^*Q))$ and the equivariant free loop space homology $H_*(\calL Q /S^1,Q)$ as graded involutive graded Lie bialgebras is sketched in \cite{cieliebak2009role}.
The chain level enhancement of this isomorphism would identify the full SFT $C^\qonly_\sftlin(T^*Q)$ of $T^*Q$ as an $\op{IBL}_\infty$ algebra (recall \S\ref{subsec:linearize}) with the same structure defined in terms of the string topology $Q$. Models for the latter have been defined in \cite{cieliebak2020homological,hajek2020ibl,cieliebak2022chain,cieliebak2023chern}, with close connections to the Cherns--Simons theory of $Q$.

Although many string topology operations (i.e. those coming from the framed $E_2$ algebra structure) are known to depend only on the homotopy type of $Q$ (see \cite{cohen2002homotopy}), recent indications suggest that this need not extend to the Goresky--Hingston string topology coproduct (see e.g. \cite{goresky_hingston,sullivan2004open,hingston_wall}), which is sensitive to the simple homotopy type of $Q$ and in particular can distinguish the homotopy equivalent lens spaces $L(7,1)$ and $L(7,2)$ (see \cite{naef2021string}).
This would suggest that a suitable version of the contact homology algebra $C_\chalin(T^*Q)$ should be strong enough to know e.g. that $T^*L(7,1)$ and $T^*L(7,2)$ are not symplectomorphic (the latter was proved by Abouzaid--Kraugh \cite{abouzaid2018simple} using Fukaya category techniques).

Going beyond the simple homotopy type, it is a great puzzle to understand what type of pseudoholomorphic invariants of $T^*Q$ could recover the full diffeomorphism type of $Q$. For instance, Eliashberg has posed the following conundrum: if $Q_1$ and $Q_2$ are homeomorphic smooth four-manifolds which are smoothly distinguished by a subtle gauge theoretic invariant such as Seiberg--Witten theory, can we find an analogous symplectic invariant which distinguish $T^*Q_1$ from $T^*Q_2$?
Note that recent developments in Floer homotopy theory (see e.g. \cite{cohen2020floer,abouzaid2024foundation,large2021spectral,cote2023equivariant}) should provide a wealth of new spectrally enriched pseudoholomorphic curve invariants which may help shed some light on this mystery.
One naturally expects parallel spectrally enriched versions of symplectic field theory, which should retain more information about higher index moduli spaces of punctured curves (e.g. via the apparatus of flow categories).
Lastly, let us point out that one can also study the action filtered version of symplectic field theory for $T^*Q$ and $S^*Q$ as in \S\ref{subsec:quant}, which is related to the energy functional on $\calL Q$ and is sensitive not just to the smooth topology of $Q$ but also to its Riemannian geometry.

\bibliographystyle{math}
\bibliography{biblio}

\end{document}